\documentclass[11pt,reqno,twoside]{amsart}
\usepackage{amsmath}
\usepackage{amssymb}
\usepackage{amsthm}
\usepackage{mathrsfs} 
\usepackage{mathtools}
\usepackage{algorithm}
\usepackage{algpseudocode}
\usepackage{algorithmicx}
\usepackage{graphicx}
\usepackage{setspace}
\usepackage[utf8]{inputenc}
\usepackage[english]{babel}
\usepackage[markup=nocolor]{changes}
\usepackage{listings}
\usepackage{color} 
\usepackage{caption}
\usepackage{subcaption}
\usepackage{cancel}
\usepackage[hidelinks]{hyperref}
\usepackage{cleveref}
\usepackage{gnuplottex}
\usepackage{enumerate}
\usepackage{multirow}

\usepackage[dvips,bottom=1.4in,right=1in,top=1in, left=1in]{geometry}

\newtheorem{theorem}{Theorem}[section]
\newtheorem{lemma}{Lemma}[section]
\newtheorem{definition}{Definition}[section]

\newcommand{\flap}{\ensuremath{(-\Delta)^s}}
\newcommand{\flapp}{\ensuremath{(\widetilde{-\Delta})^s}}
\newcommand{\R}{\ensuremath{\mathbb{R}}}
\newcommand{\I}{\ensuremath{\mathcal{I}}}
\renewcommand{\S}{\ensuremath{\mathcal{S}}}
\newcommand{\N}{\ensuremath{\mathbb{N}}}
\newcommand{\Z}{\ensuremath{\mathbb{Z}}}
\newcommand{\y}{\ensuremath{\mathbf y}}
\newcommand{\x}{\ensuremath{\mathbf x}}
\renewcommand{\u}{\ensuremath{\mathbf u}}
\newcommand{\f}{\ensuremath{\mathbf f}}
\newcommand{\id}{\ensuremath{\mathbf{1}}}
\newcommand{\eps}{\varepsilon}

\newcommand{\dom}{\mathop{d\omega}}
\newcommand{\FT}{\ensuremath{\mathop{\mathcal F}}}
\newcommand{\IFT}{\ensuremath{{}\mathcal F^{-1}}}
\renewcommand{\i}{\mathrm{i}}
\newcommand{\e}{\mathrm{e}}

\DeclareMathOperator{\supp}{supp}
\DeclareMathOperator{\sinc}{sinc}

\DeclareMathOperator{\DFT}{DFT}
\DeclareMathOperator{\IDFT}{IDFT}

\numberwithin{equation}{section}

\title{Approximation of Integral Fractional Laplacian and Fractional PDEs via sinc-Basis}

\author{Harbir Antil}
\address{Department of Mathematical Sciences and the Center for Mathematics and Artificial Intelligence (CMAI), George Mason University, Fairfax, VA 22030, USA.}
\email{hantil@gmu.edu}

\author{Patrick Dondl \and Ludwig Striet}
\address{Abteilung für Angewandte Mathematik, Albert-Ludwigs-Universität Freiburg, Hermann-Herder-Straße 10, 79104 Freiburg i. Br. }
\email{patrick.dondl@mathematik.uni-freiburg.de, ludwig.striet@mathematik.uni-freiburg.de}

\thanks{LS gratefully acknowledges a doctoral scholarship from the Friedrich-Ebert-Stiftung. HA is partially supported by NSF grants DMS-1818772, DMS-1913004, the Air Force Office of Scientific Research (AFOSR) under Award NO: FA9550-19-1-0036, and Department of Navy, Naval PostGraduate School under Award NO: N00244-20-1-0005.}

\begin{document}

\maketitle

\begin{abstract}
	Fueled by many applications in random processes, imaging science, geophysics, etc., fractional Laplacians have recently received significant attention. The key driving force behind the success of this operator is its ability to capture non-local effects while enforcing less smoothness on functions. In this article, we introduce a spectral method to approximate this operator employing a sinc basis. Using our scheme, the evaluation of the operator and its application onto a vector has complexity of $\mathcal O(N\log(N))$ where $N$ is the number of unknowns. Thus, using iterative methods such as Conjugate Gradient (CG), we provide an efficient strategy to solve fractional partial differential equations with exterior Dirichlet conditions on arbitrary Lipschitz domains. Our implementation works in both $2d$ and $3d$. We also recover the FEM rates of convergence on benchmark problems. For fractional exponent $s=1/4$ our current 3d implementation can solve the  Dirichlet problem with $5\cdot 10^6$ unknowns in under 2 hours on a standard office workstation. We further illustrate the efficiency of our approach by applying it to fractional Allen-Cahn and image denoising problems. 
\end{abstract}

\section{Introduction}
\label{sec:introduction}

This article is concerned with the numerical treatment of equations of the form
\begin{align}\label{eq:bvp}
\begin{aligned}
 	\flap u &= f \quad \text{ in } \Omega \\
 		u &= 0 \quad \text{ in }\R^d\setminus\Omega,
\end{aligned} 		
\end{align}
on an open bounded domain $\Omega \subset \R^d$ with Lipschitz boundary
$\partial\Omega$. Furthermore, we consider applications to phase field models and imaging science.

A standard way to define the fractional Laplacian is via a principal value integral on functions in the Schwartz space $\S(\R^d)$ of rapidly decaying functions.
 \begin{definition}
 	\label{def:flap_int}
	 Let $u \in \S(\R^d)$ and $s\in (0,1)$. The operator $\flap$ is defined as
 	\begin{align}
		\flap u(x) &= C(d,s) P.V.\int\limits_{\R^d} \frac{u(x)-u(y)}{|x-y|^{d+2s}}\mathop{dy} \label{eq:flap_integral} \text{ where } C(d,s) = \frac{s 2^{2s} \Gamma(s + d/2)}{\pi^{d/2} \Gamma(1-s)}
 	\end{align}
 	is a normalization constant.
 \end{definition}
 This definition clearly shows the non-local character of the operator \flap. In order to evaluate $\flap u(x)$ at a single point $x\in\R^d$, one has to evaluate the singular integral over the full space $\R^d$. Furthermore, this fractional Laplacian only makes sense for functions which are defined on all of $\R^d$. The definition can be interpreted in a weak sense for regular distributions with suitable growth conditions via integration by parts \cite[Chapter 2]{SilvestrePhD}. We will from now on use this interpretation implicitly when needed.
 
 For the remainder of this paper, we will
refer to problems of the form \cref{eq:bvp} as the Dirichlet problem for the fractional
Laplacian, or -- more explicitly -- as the fractional Poisson problem with Dirichlet exterior conditions.  
The fractional Laplacian
 of a function $u$ with support in $\Omega \subset \R^d$, $\Omega$ bounded, or, equivalently, of a function $u\colon\Omega\to\R$ which is extended by zero outside $\Omega$ will be denoted as the Dirichlet fractional Laplacian. We remark that other options to give meaning to fractional operators applied to functions defined on bounded domains exist (for an overview, see \cite{Lischke2020}), but here, we will be concerned with the Dirichlet problem as described above. 

 On the entire $\R^d$, there are at least nine further, equivalent ways to define the fractional Laplacian, see \cite{Kwasnicki2015}. A particularly useful one is the one defined using Fourier transform. See for example \cite{DiNezza2012} for a proof of the following

 \begin{theorem} Let $s \in(0,1)$ and let $\flap : \S(\R^d) \rightarrow L^2(\R^d)$
 	the fractional Laplacian from \cref{def:flap_int}. Then, for $u\in\mathcal S$,
 	\begin{equation}
 		\label{eq:flap_ft}
 		\flap u = \IFT\left( |\omega|^{2s} (\FT u) \right).
 	\end{equation}
 \end{theorem}

In principle, \cref{eq:flap_ft} gives a direct method to solve
\cref{eq:bvp} via the Fourier transformation. This is, of course, numerically intractable
as it is not possible to perform a discrete Fourier transformation (DFT) on an infinite domain.
On the other hand, if one truncates the domain on which the Fourier transformation is performed, one instead obtains the fractional Laplacian of a function $u$ that was periodically extended outside the truncation. We note that this is a different operator than the Dirichlet fractional Laplacian that we will call periodic fractional Laplacian, denoted $\flapp$.

Fractional partial differential equations (PDEs) of type \cref{eq:bvp} have recently received 
a significant amount of attention. The interest in fractional operators of type 
$(-\Delta)^s$ with $s \in (0,1)$ stems from two facts: (i) these operators impose less 
smoothness (cf.\ the classical case $s=1$); (ii) even more importantly, these operators easily
enable nonlocal interaction, recall that the classical derivatives lead to local operators.

The fractional Laplacian has been used as a regularizer in imaging \cite{AntilBartels2017,Antil2019}. 
Additionally, it can be derived using the so called long jump random walk \cite{EValdinoci_2009a}. 
Fractional diffusion-reaction equations such as fractional Allen-Cahn have been
studied in \cite{Akagi2015, Alberti:1994ui, Savin:2012fl, AntilBartels2017},
and a fractional Cahn-Hilliard equation has been studied e.g. in \cite{Mao2017,
Akagi2015, AntilBartels2017}. A related example where, instead of the
space-derivative the time-derivative is taken to fractional order in a
diffusion-reaction equation is e.g. \cite{Jin2015}.
Finally, the fractional Helmholtz equation has been recently derived in \cite{weiss2020fractional} using 
the first principle arguments in conjunction with a constitutive relation. We emphasize that 
the fractional Laplacian in \cite{weiss2020fractional} is of the so-called spectral type. 
From an optimal control point of view, fractional operators provide a great deal of flexibility,
since the condition $u = 0$ is imposed in the exterior $\R^d\setminus \Omega$ of the domain $\Omega$.
Therefore it has been possible to introduce a new type of optimal control (called Exterior Control) using fractional PDEs \cite{Antil2018b}.

It has been noted in \cite{Caffarelli2006} that -- in $\R^d$ -- the definitions \cref{eq:flap_integral} and \cref{eq:flap_ft} 
are further equivalent to the so called extension problem. However, even for functions $u$ such that $\supp u \subset \Omega$, with $\Omega$ bounded, the extension has to be performed on $\R^d\times (0,\infty)$ to compute the fractional Laplacian for the Dirichlet problem. It is still possible to define an extension problem on $\Omega\times (0,\infty)$ which is equivalent to the so-called 
spectral fractional Laplacian \cite{Stinga2010}. Using such an extension, efficient finite element based numerical methods have been
proposed in \cite{MR3348172,DMeidner_JPfefferer_KSchurholz_BVexler_2018a}.
For completeness, we also refer to \cite{HAntil_CNRautenberg_2019a} for an extension problem where the fractional exponent
$s$ is a function of the spatial variable $x \in \Omega$. 

On the other hand, the numerical methods for problems of the type \cref{eq:bvp} present even more challenges, as one needs 
to resolve the singular integrals. The first work that rigorously tackles numerics for \cref{eq:bvp} using finite element method is by Acosta and Borthagaray \cite{Acosta2017}, see also 
\cite{Ainsworth2018}. However, the implementations in these works have been limited to $d = 2$ dimensions.
We also refer to another finite element approach of Bonito, Lei, and Pasciak \cite{ABonito_WLei_JEPasciak_2019a}, which also works in $d=3$ \cite{MR3893441}.

In contrast, we provide a spectral method to approximate the Dirichlet fractional Laplacian \cref{eq:flap_integral} where we attempt to combine fast Fourier transform (FFT) efficiency with the ability to treat exterior value problems as in \cref{eq:bvp}. As we shall illustrate with numerical examples, this approach directly applies to the case when $d = 3$, and extension to even higher dimensions is possible given efficient FFT implementations. Using our method, the application of the fractional operator has the same numerical complexity as a Fourier transform. To solve the Dirichlet problem \cref{eq:bvp} we then use this operator, restricted to $\Omega$, within a conjugate gradient algorithm. 

Related strategies have been suggested recently by other authors. Duo and Zhang introduced a finite difference scheme to efficiently solve equations involving the fractional Laplace operator \cite{Duo2018}. Their method relies on a finite difference approximation of the operator and the fact that this can be expressed as a matrix consisting of blocks of symmetric Toeplitz-Matrices which can be applied efficiently using discrete Fourier transformation based methods. Given this operator, they solve fractional PDEs using iterative methods. Minden and Ying introduced a method to discretize the integral operator \cref{eq:flap_integral} which also leads to a Toeplitz-Matrix. Along with a preconditioner, they also use the conjugate gradient method to solve the arising systems for the fractional Dirichlet problem and fractional diffusion equations \cite{Minden2020}. 
Another example for a spectral method was provided recently by Xu and Darve.
They use eigenfunctions and eigenvalues w.r.t a specific weight of the
Dirichlet fractional Laplacian to solve the fractional Dirichlet problem on the
unit ball \cite{Xu2018}.

Recently,  mesh-free methods based on radial basis functions (RBFs)
to estimate fractional Laplacians \cite{Burkardt2020, Rosenfeld2019} have been proposed. 
Those methods
are related to our method in the sense that they also rely on the element-wise
application of an operator to basis functions. Further, they use the fact that the
Fourier transformation of smooth functions has rapid decay. A novelty that our method shares with \cite{Burkardt2020}
is that we can treat classical and fractional PDE in a single framework, as illustrated in the experiments in the experiments
in \cref{sub:const_rhs}.

A more general, unified framework for finite difference schemes to evaluate Fractional Laplacians for $d=1$ is presented by Huang and Oberman in \cite{Huang2016}. In particular, they use a finite difference scheme based on $\sinc$-interpolation and show an equivalence to a Fourier scheme on an inifinite interval. We extend those results to multiple dimensions and present numerical schemes to calculate the resulting integrals. 

The article is organized as follows. In \cref{sec:sincfrac} we briefly discuss some properties of the $\sinc$-interpolation used in the present work as well as a discrete version of the fractional Laplacian based on this interpolation. The numerical methods to efficiently evaluate fractional Laplacians, as well as some equivalences between scaled periodic and our discretized operators, are presented in \cref{sec:num_methods}.
In a supplementary document, we provide more details on the implementation of our algorithms and on
the computational complexity, see \cref{sub:algorithms}. \Cref{sec:numexp} is devoted to validation of our theoretical results using various numerical examples in both $d=2$ and $d=3$. 
The $L^2$-convergence rates for finite element methods as stated in
\cite{Ainsworth2018} are obtained also using our method.
We conclude the article by successfully applying the proposed scheme to the fractional Allen-Cahn equation and image denoising problems.

\section{The sinc-Fractional Laplacian}
In this section, we define our discrete approximation of the Dirichlet fractional Laplacian. First, in \cref{sec:interpolation}, we introduce an appropriate discretization method for functions with bounded support. Following this, we use a discrete convolution to define our operator in \cref{sec:discrete_operator}.

\label{sec:sincfrac}
\subsection{sinc-Interpolation}\label{sec:interpolation}
It is a well-known fact that, for example, smooth functions with compact support on $\R$ can be well approximated by a weighted sum of scaled and shifted $\sinc$-functions. The $\sinc$-function is defined as
\begin{equation}
	\sinc(x) = \frac{\sin(\pi x)}{\pi x}.
\end{equation}
Two different approaches of approximating functions with $\sinc$-functions are found in the literature. The first approach is to consider the $\sinc$-function as a wavelet scaling function and then to approximate in the sense of wavelets which leads to the so called Shannon-wavelets. A good overview on such techniques can be found in \cite{Cattani2008}. They have also been used to approximate fractional derivatives as shown in \cite{Cattani2016, Cattani2018}.

The second approach, which we will pursue here, is to approximate a function ${u : \R\to\R}$ by
\begin{equation}
	\label{eq:sinc_approx}
	u(x) \approx \sum\limits_{k=-\infty}^{\infty} u(kh) \sinc\left( \frac{x-kh}{h} \right).
\end{equation}
The sum on the right hand side (if it converges) is called \textit{Whittaker Cardinal Function}, see, e.g.,  \cite{McNamee1971} for more details. This approximation is rather precise \cite{Stenger2010} and has found numerous applications, e.g., for solving ordinary differential equations, partial
differential equations and integral equations \cite{Stenger2000, Stenger2010},
but also for approximating integrals arising in fractional calculus
\cite{Baumann2011}. Furthermore, this is the method chosen in \cite{Huang2016}
to derive the weights for a finite difference scheme to compute fractional
Laplacians in one dimension. Another way the $\sinc$-functions are employed
in the area of fractional PDEs is to use $\sinc$-quadrature to compute the
Dunford-Taylor integral arising from the fractional Laplacian, as done, e.g.,
in \cite{ABonito_WLei_JEPasciak_2019a}.

In this work, we will use approximations similar to \cref{eq:sinc_approx},
extended to multiple dimensions. For $x = (x_1, \ldots, x_d)^T\in \R^d$, we set
\begin{equation}
	\varphi(x) = \prod\limits_{j=1}^d \sinc(x_j)
\end{equation}
as the reference basis function and define its scaled and shifted version ${\varphi^N_k \colon \R^d\to\R}$
for $k=(k_1,\dots,k_d)^T\in\Z^d$ by a tensor product
\begin{align}
	\varphi_k^N(x) = \prod\limits_{j=1}^d \varphi(Nx_j - k_j),\quad x = (x_1,\ldots,x_d)^T\in\R^d
\end{align}
where $\Z^d$ is the $d$-dimensional integer lattice.
For $u : \R^d \longrightarrow \R$, we obtain the sinc approximation
\begin{equation}
	u_N(x) = \sum\limits_{k\in\Z^d}u_k\varphi^N_k(x)\quad\text{where }u_k = u(x_k), x_k = k/N.
\end{equation}
Since we are interested in computing the Dirichlet fractional Laplacian, our functions have compact support, which we from now on assume to be contained in the unit cube $[0,1)^d$ (the generalization to other, larger domains is trivial; our assumption is merely for notational convenience). We therefore may truncate the series to 
\begin{equation*}
	u_N(x) =
	\sum\limits_{k_1=0}^{N-1}\cdots\sum\limits_{k_d=0}^{N-1}u_k\varphi^N_k(x) 
		\eqqcolon \sum\limits_{k\in\I_N^d}u_k\varphi^N_k(x),
\end{equation*}
where, if $\I_N \coloneqq \{0,\ldots,N-1\}$, then 
$\I_N^d\subset\Z^d$ indicates its $d$-fold Cartesian product. Later in this work, we will also use the set $\I_N' = \{-N/2,\ldots,N/2-1\}$ and its Cartesian, product respectively.

It is a well-known fact that the basis function $\varphi$ can be obtained as
the inverse Fourier transform of the indicator function 
of a square, i.e.,  
for ${D = [-\pi;\pi]^d}$ we have
\begin{align*}
		(\IFT \chi_D(\omega))(x) 
		 = (2\pi)^d\varphi(x)
\end{align*}
where $ \chi_D(x) = 1 \text{ if } x\in D, 0\text{ otherwise}$ Similarly, $\varphi^N_k(x)$ can be obtained as 
\begin{equation*}
	\varphi^N_k(x) = \varphi(Nx-k)
		= \IFT \underbrace{\left((2\pi N)^{-d} \chi_{D_N}(\omega) e^{-\i \omega k/N}\right)}_{=\FT \varphi^N_k(\omega)}(x)
\end{equation*}
where $D_N = [-N\pi;N\pi]^d\subset \R^d$.
\subsection{The discrete operator}
\label{sec:discrete_operator}
Given a real function $u$ with $\supp u \subset [0;1)^d$, we want
to apply the Dirichlet Fractional Laplace operator $\flap$ of
\cref{eq:flap_ft} to its $\sinc$-approximation $u_N$. We shall write
\begin{equation*}
	\flap_N u \coloneqq \flap u_N
\end{equation*}
and call the operator $\flap_N$ the \emph{sinc-fractional Laplacian}.
We are restricted to grid points, so let $x_\kappa = \kappa/N$, $\kappa \in \I_N^d$. As $\I_N^d$ is
finite and $\flap$ is linear, we obtain
\begin{align}
	\flap u_N(x_\kappa) &= \flap \Big(\sum\limits_{k\in\I_N^d}u_k \varphi^N_k(x_\kappa)\Big) \nonumber\\
			&=  \sum\limits_{k\in\I_N^d}u_k \underbrace{\left(\flap\varphi^N_k\right)(x_\kappa)}
			_{\eqqcolon\Phi^N(\kappa-k)}\nonumber \\
			&=  \sum\limits_{k\in\I_N^d}u_k \Phi^N(\kappa-k).\label{eq:PHI_1}
\end{align}
where we have defined $\Phi^N(\kappa-k) = \flap\varphi^N_k(x_\kappa)$ for $k,\kappa\in\I_N^d$.
In the remainder of this paper, we will occasionally use the notation $\Phi^N_K
= \Phi^N(\kappa-k)$, with $K = \kappa-k$ when it is clear from the context. Notice that \cref{eq:PHI_1} denotes a discrete convolution. The computation of this convolution is the application of the sinc-fractional Laplacian. 

In other words, we obtain the sinc-fractional Laplacian $\flap u_N(x_\kappa)$ for any grid point
$x_\kappa$ as the discrete convolution of $\mathbf u = (u_k)_{k\in\I_N^d}$ and
$\Phi^N$.
Such a convolution can be implemented efficiently using the FFT algorithm once $\Phi^N(\kappa-k)$ is known for all $k,\kappa\in \I_N^d$. More precisely, the \textit{circular} discrete
convolution of two vectors $\x, \y \in \R^{N^d}$ can be calculated as 
\begin{align}
	\big(\x*_d \y\big)(k) &\coloneqq \sum\limits_{\kappa\in\I_N^d} \bar{\mathbf x}(\kappa)\cdot \bar{\mathbf y}(k-\kappa) \label{eq:disc_conv_def}\\
		&= \IDFT\left\{ (\DFT \x) \circ (\DFT \y) \right\}(k)\label{eq:fast_disc_conv}
\end{align}
where $\circ$ denotes the component-wise product of vectors, and $\DFT$ and $\IDFT$
denote the discrete Fourier transformation and the inverse discrete Fourier
transformation, respectively. By circular we mean that negative
components of indices $\kappa-k$ are mapped circularly to their positive counterparts,
in formulas
\[
	\bar{\x} (K) = \begin{cases}
		\x(K)&\text{if }K \geq 0 \\
		\x(K+N)&\text{if }K < 0.
	\end{cases}
\]
While the evaluation of \cref{eq:disc_conv_def} is of complexity 
$\mathcal O((N^d)^2)$, if evaluated for each $k$, the simultaneous evaluation of
\cref{eq:fast_disc_conv} for all $k$ can be  implemented in  $\mathcal O(N^d
\log(N^d))$ time.

For our application, we do not actually want to apply the circular convolution,
but the convolution where we extend by zero instead of periodically. As we still
want to use the FFT-based algorithm to evaluate \cref{eq:PHI_1} because
of its computational efficiency, we set
\[
	\bar u_k = \begin{cases}
		u_k &\text{if } k \geq 0 \\
		0 & \text{otherwise}
	\end{cases}
\]
where the expression $k \geq 0$ is meant component-wise. Then, we have that
\[
	\flap u_N(x_\kappa) = \big(\bar u *_d \Phi^N\big)(\kappa)
\]
which we implement using a FFT of size $(2N)^d$. Further details on the implementation and pseudocode can be found in \cref{sub:algorithms}. 

We have shown that we can obtain the basis function $\varphi(\cdot)$ as the 
inverse Fourier transformation of the indicator function of a square in $\R^d$.
This can in principle be used to obtain the integral fractional Laplacian
$\left(\flap\varphi\right)$ of the basis functions, since for $x\in\R^d$, we have 
\begin{align}
	 \big(\flap\varphi\big)(x) &= \IFT\left( |\omega|^{2s} (\FT \varphi) \right)\nonumber \\
		&= (2\pi)^{-d} \int_D |\omega|^{2s} e^{\i\omega\cdot x}\dom,
\end{align}
and
\begin{align}
	\Phi^N(\kappa-k) &= (2\pi N)^{-d} \int_{D_N} |\omega|^{2s} e^{\i \omega\cdot (x_\kappa - k/N) }\dom \nonumber\\
	&= (2\pi)^{-d} N^{2s} \int_D |\omega|^{2s} e^{\i\omega\cdot(\kappa-k)}\dom\nonumber\\
	&= N^{2s} \big(\flap\varphi\big)(\kappa-k). \label{eq:PHI_integral}
\end{align}

However, using this equality directly is impractical as we would have to evaluate
the oscillating integral for each multi-index $k$. For the one dimensional
case, it is possible to circumvent this issue through the use of the confluent
hypergeometric function \cite{Huang2016}, but for $d>1$ a numerical solution must be found.
Clearly, calculating $\Phi^N$ itself is not necessary in order to implement
\cref{eq:PHI_1} as we only need the discrete Fourier transformation $\hat\Phi^N$ of $\Phi^N$. In \cref{sub:setting_up_the_conv_kernel} we show how  $\hat\Phi^N$ can be obtained efficiently.

\section{Numerical Methods}
\label{sec:num_methods}
The goal of this section is to introduce our numerical methods.
We begin with \cref{sub:calculation_periodic_flap} where we discuss the  computation of fractional Laplacians using simple Fourier methods as mentioned in the introduction. In \cref{sub:equivalence}, we show that the sinc-fractional Laplacian applied to a function with support in $[0,1)^d$, as defined in \cref{eq:PHI_1}, can be seen as a limit $S\to \infty$ of fractional Laplacians obtained by standard Fourier transforms of the function periodically extended outside $[0,S)^d$.

We then describe, in \cref{sub:setting_up_the_conv_kernel}, our numerical quadrature method used to compute $\hat\Phi$. The solution strategy to the Dirichlet exterior value problem \cref{eq:bvp} is discussed in \cref{sub:solving_the_system}.

\subsection{Computation of the periodic fractional Laplacian}
\label{sub:calculation_periodic_flap}

To experimentally test that our discrete approximation \cref{eq:PHI_1} indeed approaches the Dirichlet fractional Laplacian, we can compare it to the periodic fractional Laplacian which is calculated
by extending $u$ periodically outside of a truncation domain, instead of extending $u$ by
$0$. 

For periodic functions, the periodic fractional Laplacian and 
the integral fractional Laplacian are equal \cite{NAbatangelo_EValdinoci_2019a}. If the Dirichlet fractional Laplacian
is approximated with the periodic fractional Laplacian on a finite domain, an error 
is introduced due to the implicit periodization of the function. This effect, however,
is reduced if the function is scaled with a factor $S > 1$ before the
application of the periodic fractional Laplacian and rescaled to the original
domain afterwards. Heuristically, this occurs because the additional support introduced by the periodic continuation becomes shifted further away from the original support. The error is then on the order of $S^{-(d+2s)}$. This estimate is summarized in the following Lemma. The main motivation for this Lemma is the fact that in the next section we shall establish an equivalence between the scaled periodic fractional Laplacian and the sinc-fractional Laplacian, see \cref{thm:equiv_scaledperi_shannon,thm:equiv_scaledperi_shannon_2}.

\begin{lemma}\label{thm:diff_flappS_flap}
Let $u\in L^1(\R^d)$ with $\supp u$ \sout{$= \Omega$} $\subset [0, 1)^d$. Let $\flap$ be the Dirichlet fractional Laplacian (see
	equation \cref{eq:flap_integral}) and $\flapp$ be the periodic fractional Laplacian (applied on the function restricted to $[0, 1)^d$). Assume furthermore that $\flap u \in L^1([0,1)^d)$.
     Then, for a.e. $x\in(0,1)^d$, $S$ sufficiently large, we have
	\begin{equation}
	\label{eq:flap_spectral_approaches_int}
		 S^{-2s} \left(\flapp u(S\,\cdot)\right)(x/S) = \flap u(x) + \mathcal O\left(S^{-(d+2s)}\right).
	\end{equation}
\end{lemma} 
\begin{proof}
Assume first that $u\in C_c^\infty([0,1)^d)$, fix $x\in(0,1)^d$ and set $\tilde x = x/S$, $u_S=u(S\cdot)$. We calculate  
\begin{align}
		\flapp u_S(\tilde x) &= \flap\left(\sum\limits_{k\in\Z^d} u_S(\cdot-k)\right)(\tilde x) \nonumber\\
			&= C(d,s) \int\limits_{\R^d} \frac{ u(S\tilde x) -  u(Sy)}{|\tilde x-y|^{d+2s}}\mathop{dy} \label{eq:flap_spectral_approaches_int_a}\\
				&\qquad+ C(d,s) \sum_{k\in\Z^d\setminus\{0\}}\;\int\limits_{\R^d} \frac{ u(S\tilde x-Sk) -  u(Sy-Sk)}{|\tilde x-y|^{d+2s}}\mathop{dy},\label{eq:flap_spectral_approaches_int_b}
		\end{align}

		where we note that the sum in equation \eqref{eq:flap_spectral_approaches_int_b} converges absolutely, owing to  \eqref{eq:estimate_int_decay1} and the estimate \eqref{eq:estimate_int_decay2} below.

	A linear transformation in the term in \eqref{eq:flap_spectral_approaches_int_a} yields
\[
C(d,s) \int\limits_{\R^d} \frac{ u(S\tilde x) -  u(Sy)}{|\tilde x-y|^{d+2s}}\mathop{dy}
 = C(d,s) \int\limits_{\R^d} \frac{ u(S\cdot x/S) -  u(Sy)}{|x/S-y|^{d+2s}} \frac{S^d}{S^d}  \mathop{dy}
 = S^{2s} \flap u(x).
\]
	For the integral in the summand in \eqref{eq:flap_spectral_approaches_int_b}, we have for $k\neq 0$
	\begin{equation} \label{eq:estimate_int_decay1}
		\int\limits_{\R^d} \frac{\left[ u(S\tilde x-Sk) -  u(Sy-Sk)\right]}
			{|\tilde x-y|^{d+2s}}\mathop{dy} = 
-\int\limits_{\R^d} \frac{u(Sy-Sk)}
			{|\tilde x-y|^{d+2s}}\mathop{dy}
	\end{equation}
	as $S\tilde x - Sk \not\in (0,1)^d \supset \supp u$.
	Furthermore,
	\begin{equation} \label{eq:estimate_int_decay2}
		\left|\,\int\limits_{\R^d} \frac{   u(Sy-Sk)} {|\tilde x-y|^{d+2s}}\mathop{dy}\right|
		=
		S^{2s}\int\limits_{(0,1)^d} \frac{u(y)}{|x - y - Sk|^{d+2s}}dy \leq
		\|u\|_{L^1((0,1)^d)}\cdot C' \frac{1}{S^d} \frac{1}{|k|^{d+2s}},	
	\end{equation}
	 where the constant $C'$ remains bounded for large $S$. 
	Summing over $k\neq 0$ and dividing by $S^{2s}$ yields  the  result.
	
For $\flap u \in L^1$ only, note that we still have $\flapp u_S(S\cdot) =S^{2s} \flap u(\cdot)  -  \sum_{k\in\Z^d\setminus\{0\}}\,\int\limits_{\R^d} \frac{u(Sy-Sk)}
			{|S\cdot -y|^{d+2s}}\mathop{dy}$ and that the second term is a non-singular integral that can be estimated as above.
\end{proof}

The periodic fractional Laplacian can be discretized using the DFT. 
A comprehensive overview is provided in \cite{AntilBartels2017}. Briefly repeated,
the $N^d$-point discrete Fourier transformation of a vector $\x\in\R^{N^d}$ is defined as
\begin{equation}
\label{eq:def_DFT}
	\left(\DFT_N \x\right)_k = \hat x_k = \sum\limits_{j_1=0}^{N-1}\cdots\sum\limits_{j_d=0}^{N-1}
		x_j \e^{-\i\frac{ 2\pi j\cdot k}{N}}
\end{equation}
and the inverse discrete Fourier transform is
\begin{equation}
\label{eq:def_IDFT}
	\left(\IDFT_N \hat \x\right)_k = x_k = \frac{1}{N^d} \sum\limits_{j_1=0}^{N-1}\cdots\sum\limits_{j_d=0}^{N-1}
		\hat x_j \e^{\i\frac{ 2\pi j\cdot k}{N}}.
\end{equation}
If the size $N$ of the DFT is obvious, we will omit the subscript index $N$.
The $N^d$-point discrete periodic fractional Laplacian of $u : [0;1)^d
\longrightarrow \R$ is calculated via
\[
	\left(\flapp_{N}u\right) (x_\kappa) = \big(\IDFT_N( \zeta \circ \DFT_N u)\big) (\kappa)
\]
where $\circ$ denotes the Hadamard (entrywise) product, $\DFT$ and $\IDFT$ the $N^d$-point discrete (inverse) Fourier transformation and 
\[
	\zeta_k = |2\pi k|^{2s}.
\]
Note that usual FFT implementations of the DFT calculate the discrete spectrum
of $f$ in the range $\{0,\ldots,N-1\}^d$. To have the factors $\zeta$ at the
correct scale, one has to shift the Fourier coefficients periodically to the
interval $\{-N/2, \ldots, N/2-1\}$ to obtain
\begin{align*}
	\left(\flapp_{N}u\right) (x_\kappa) &= \big(\IDFT_N( \zeta \circ \DFT_N u)\big) (\kappa) \\
		&= \frac{1}{N^d} \sum\limits_{k_1=-N/2}^{N/2-1}\cdots\sum\limits_{k_d=-N/2}^{N/2-1}
			 |2\pi k|^{2s} \hat u_k \e^{i \frac{2\pi k\cdot \kappa}{N}} .
\end{align*}

If $S\in\N$, we calculate the discretized scaled periodic fractional Laplacian in
\cref{eq:flap_spectral_approaches_int} similarly using the $SN$-point DFT and
inverse DFT as follows: we extend the vector
\[
	\textbf u = (u_k)_{k\in\I_N^d} \in \R^{N^d}, u_k = u(k/N)
\]
to a vector
\[
	\bar {\textbf u} = (\bar u_k)_{k\in\I_{SN}^d}\in \R^{(SN)^d}, \bar u_k = \begin{cases}
		u_k &\text{if } k < N \\
		0&\text{otherwise}
	\end{cases}
\]
where the $<$-sign is meant component-wise and obtain
\begin{align}
	\left(\flapp_{S,N}u\right) (x_\kappa) &=  \big(\IDFT_{SN}( \zeta \circ \DFT_{SN} u)\big) (\kappa) \\
		&= \frac{1}{(SN)^d} \sum\limits_{j_1=-SN/2}^{SN/2-1}\cdots\sum\limits_{j_d=-SN/2}^{SN/2-1}
			 |2\pi j|^{2s} \hat u_j\e^{\i \frac{2\pi j\cdot \kappa}{SN}} \\
		&=  \frac{(2\pi)^{2s}}{(SN)^d}\sum\limits_{j\in{\I'}_{SN}^d} |j|^{2s}
			\left(\sum\limits_{k\in\I_{SN}^d} u_k \e^{-\i \frac{2\pi j\cdot k}{SN}  } \right)
				\e^{\i \frac{2\pi j\cdot \kappa}{SN}} \\
		&=  \frac{(2\pi)^{2s}}{(SN)^d} 
		\sum\limits_{k\in\I_N^d} u_k\sum\limits_{j\in{\I'}_{SN}^d}
		|j|^{2s} \e^{\i \frac{2\pi}{SN} j\cdot (\kappa -k)}. \label{eq:expl_form_scaled_periodic}
\end{align}
The last line is certainly not the most efficient way to evaluate $\left(\flapp_{S,N}u\right) (x_\kappa)$ -- this
 should instead be done via the FFT algorithm as stated in the first line.
However, the expression will be needed in the following to show an equivalence of the sinc-fractional Laplacian and the scaled periodic fractional Laplacian.
\subsection{Equivalence of the scaled periodic fractional Laplacian and the sinc-fractional Laplacian}
\label{sub:equivalence}
There is a certain equivalence of the discrete scaled periodic fractional
Laplacian and the sinc-fractional Laplacian. Precisely, if the integration in 
the calculation of $\Phi$, see \cref{eq:PHI_integral} is done exactly,
the $N^d$-point sinc-fractional Laplacian is the same as the $N^d$-point
discrete scaled periodic fractional Laplacian with infinite scale factor. To be precise, we have the following theorem.
\begin{theorem}\label{thm:equiv_scaledperi_shannon}
	Let $u\in C_c([0;1]^d)$.
	Then, for $S\longrightarrow\infty$ 
	\[
		S^{-2s}\big(\flapp_{S,N}u\big) (x_\kappa) \longrightarrow\big(\flap_{N}u\big) (x_\kappa)
		\quad\forall x_\kappa = \kappa/N, \kappa \in \mathcal{I}_N^d.
	\]
\end{theorem}
\begin{proof}
	Refer to  equations \eqref{eq:expl_form_scaled_periodic}, \eqref{eq:PHI_1} to see
	that it is enough to show that
	\begin{equation}
		\label{eq:PHI_scaled_int_expr}
		S^{-2s}\frac{(2\pi)^{2s}}{(SN)^d}\sum\limits_{j\in{\I'^d_{SN}}} |j|^{2s} \e^{\i \frac{2\pi}{SN} j\cdot (\kappa -k)}
			\longrightarrow \Phi^N(\kappa-k)
	\end{equation}
	$\forall K\coloneqq \kappa - k$ as $S\longrightarrow \infty$. Indeed, we have
	\[
		\Phi^N(\kappa-k) = (2\pi N)^{-d}\int_{D_N} |\omega|^{2s} \e^{\i \frac{\omega \cdot (\kappa-k)}{N} }\dom
	\]
	(see equation \eqref{eq:PHI_integral}) and for the left-hand side of
	\eqref{eq:PHI_scaled_int_expr}, we have
	\begin{align*}
		&\frac{(2\pi)^{2s}}{N^d}S^{-(d+2s)}   \sum\limits_{j\in{\I'^d_{SN}}} |j|^{2s} \e^{\i \frac{2\pi}{SN} j\cdot K} \\
		=&\frac{(2\pi)^{2s}}{N^d}S^{-(d+2s)}\sum\limits_{j\in{\I'^d_N}}\sum\limits_{i\in{\I^d_S}}
			|Sj + i|^{2s}\e^{\i \frac{2\pi}{SN}(Sj + i)\cdot K} \\
		=&\frac{(2\pi)^{2s}}{N^d}\sum\limits_{j\in{\I'^d_N}}\sum\limits_{i\in{\I^d_S}}
			S^{-d}|j + i/S|^{2s}\e^{\i \frac{2\pi}{N}(j + i/S)\cdot K} \\
		\overset{S\rightarrow\infty}{\longrightarrow}&\frac{(2\pi)^{2s}}{N^d}\sum\limits_{j\in{\I'^d_N}}\int\limits_{j_1}^{j_1+1}\cdots \int\limits_{j_d}^{j_d+1}
		 |\omega|^{2s}\e^{\i \frac{2\pi}{N}\omega\cdot K} \label{eq:PHI_scaled_int_expr:S_limit}\\
		=&\frac{(2\pi)^{2s}}{N^d}\int\limits_{\left[-\frac{N}{2}  ; \frac{N}{2}\right]^d}
			|\omega|^{2s}\e^{\i \frac{2\pi}{N}\omega\cdot K}\dom\\
		=&  (2\pi N)^{-d}\int\limits_{D_N}
			|\omega|^{2s}\e^{\i\frac{\omega\cdot K}{N}}\dom
	\end{align*}
	which completes the proof.
\end{proof}

A similar result is presented in \cite{Huang2016} for functions
with non-compact support using the semi-discrete Fourier transformation. In
\cref{sub:setting_up_the_conv_kernel}, see
\cref{thm:equiv_scaledperi_shannon_2}, we show a direct relation between simple quadrature rules to evaluate the discrete convolution kernel $\hat{\Phi}$ and scaled Fourier fractional Laplacians.
\subsection{Setting up the convolution kernel}
\label{sub:setting_up_the_conv_kernel}
As stated before, we aim to calculate the discrete Fourier transformation
$\hat\Phi$ of $\Phi$ directly, instead of having to calculate it as the DFT of
$\Phi$ as the latter is hard to obtain. While the fast implementation
of the convolution is standard and can be found in many textbooks, our contribution is the formulation
that makes the use of fast convolution algorithms applicable.
Therefore, let ${{\I'}_{2N}^d = \{-N,\ldots,(N-1)\}^d}$
and ${k\in {\I'}_{2N}^d}$ a multiindex. Let $\Phi_k = \flap
\varphi_N(x_k)$, $x_k = k/N\in\R^{(2N)^d}$. Let $\hat\Phi = \DFT_{2N}(\Phi)$ the discrete
Fourier transform of $\Phi \in \R^{(2N)^d}$. We start the computation with the fact that
\begin{align}
	\hat\Phi_k &= \sum\limits_{j\in\I'^d_{2N}}\Phi_j \e^{-\i 2\pi k\cdot j/(2N) } \\
	&= N^{2s}(2\pi)^{-d} \sum\limits_{j\in{\I'}_{2N}^d}\left( \int_{D}|\omega|^{2s} \e^{ \i\omega\cdot j }\dom \right) 
		\e^{-\i \pi k\cdot j/N} \\
	&= N^{2s}
		(2\pi)^{-d} \int_{D}|\omega|^{2s}\sum\limits_{j\in\I'^d_{2N}}  
		\e^{\i j\cdot \left(\omega - \frac{\pi}{N} k\right)}\label{eq:deriv_PHI_hat_1} \dom,
\end{align}
where we use the definition of the discrete Fourier transformation in the first
equation, see \cref{eq:def_DFT} and the formula for $\Phi_j$, see
\cref{eq:PHI_integral}, in the second equation.  For $x\in\R$, define
\begin{equation}
\label{eq:Y_ident}
	Y(x)\coloneqq\sum\limits_{j=-N}^{N-1} \e^{\i j x} = \begin{cases}
		\frac{\e^{-\i Nx}(\e^{2\i Nx} - 1)}{\e^{\i x}-1}&\text{if } \e^{\i x} -1 \neq 0 \\
		2N&\text{otherwise} .
	\end{cases}
\end{equation}
To simplify the sum in \cref{eq:deriv_PHI_hat_1}, we observe that for
${x\in\R^d}$, we have
\begin{align*}
	\sum\limits_{j\in\I'^d_{2N}} \e^{\i j\cdot x} &=
			\sum\limits_{j_1=-N}^{N-1}\cdots\sum\limits_{j_d=-N}^{N-1} \e^{\i(j_1 x_1+\cdots +j_d x_d)} \\
			&=\prod\limits_{i=1}^d Y(x_i)\\
			&\eqqcolon Y_d(x),
\end{align*}
plug this into \cref{eq:deriv_PHI_hat_1} and obtain
\begin{align}
	\hat\Phi_{k} &=N^{2s}
		(2\pi)^{-d} \int_{D}|\omega|^{2s}\sum\limits_{j\in\I'^d_{2N}}  
		\e^{\i j\cdot (\omega -\pi/N k)} \dom \\
	&=N^{2s}(2\pi)^{-d}\int_{D}|\omega|^{2s}Y_d\left( \omega -\pi/N k \right)\dom \\
	&=\underbrace{N^{2s}(2\pi)^{-d}\left(\frac{\pi}{N}\right)^{d+2s}}_{=\pi^{2s}\cdot (2N)^{-d}}  
		\int_{[-N;N]^d} |\omega|^{2s}Y_d\left( \frac{\pi}{N} (\omega - k) \right)\dom .
\end{align}
Now, we have the same domains for $k$ and $\omega$. Finally, we note that
the second factor $Y_d(\cdots)$ in the integrand is periodic with the length $2N$ of
the integrals and, thus, \cref{eq:deriv_PHI_hat_1} can be implemented as
a convolution using the FFT algorithm and using quadrature rules as follows:
\begin{align}
	\hat\Phi_{k} &=\underbrace{ \pi^{2s}(2N)^{-d}}_{C(N,d,s)}
		\int_{[-N;N]^d}\left| \omega\right|^{2s}
		Y_d\left(  \frac{\pi}{N} \left(\omega -k\right)\right)\dom\\
		&\approx C(N,d,s)\sum\limits_{j\in\I'^d_{2N}}\sum\limits_{i=1}^{N_Q}\alpha_i\left| j+x_i\right|^{2s}
		Y_d\left(  \frac{\pi}{N} \left(j+x_i -k\right)\right)\\
		&= C(N,d,s)\sum\limits_{i=1}^{N_Q}\alpha_i\sum\limits_{j\in\I'^d_{2N}}\left| j+x_i\right|^{2s}
		Y_d\left( - \frac{\pi}{N} \left(k-j\right) + \frac{\pi}{N}x_i\right) .\label{eq:quadrature_rules_phi_hat}
\end{align}
Where $(x_i, \alpha_i)_{i=1,\ldots,N_Q}$ is a quadrature rule on $[0;1]^d$.
The inner sum can be obtained as a discrete convolution for each $i$ using two
forward and one backward DFTs. 
In summary, we have to execute $3N_Q$ DFTs of size $(2N)^d$ to obtain $\hat\Phi$.
This is the computationally most demanding step in our algorithm, but it has to be
performed only once when $\hat\Phi$ is applied multiple times. The values of $\hat\Phi$ could even be
stored for given of $N$ and $s$. It is also possible to reduce the sizes
of the DFTs to $(2N-1)^d$. This, however, complicates the preceding computations. 

A consequence of \cref{thm:equiv_scaledperi_shannon} is that the
exactness of the integration in \cref{eq:quadrature_rules_phi_hat} is
decisive for the accuray of our method. A trivial choice for the integration
points $x_i$ and weights $\alpha_i$ is 
\begin{equation}
\label{eq:naive_quadrature}
	x_i = \frac{i}{N_Q}, i \in \{0,\ldots,N_Q-1\}^d\text{ and }
		\alpha_i = \alpha = \frac{1}{N_Q^d}.
\end{equation}
For this quadrature rule, the sinc-fractional Laplacian is exactly the same as
the discrete scaled periodic fractional Laplacian with $S = 2N_Q$. We will
discuss more possibilities along with numerical experiments in
\cref{sub:experiments_integration_phi_hat}.
\begin{theorem}
	\label{thm:equiv_scaledperi_shannon_2}
	If $\hat\Phi$ in \cref{eq:quadrature_rules_phi_hat} is calculated using
	the quadrature rule from \cref{eq:naive_quadrature}, then with $S = 2N_Q$
	we have for all $\kappa$ that
	\[
		(\flap_{N}u)(x_\kappa) = (\flapp_{S,N}u)(x_\kappa)
	\]
\end{theorem}
\begin{proof}
	
	We already derived that
	\[
		(\flapp_{S,N} u)(x_\kappa) = \frac{(2\pi)^{2s}}{S^{2s}{NS}^d} \sum\limits_{k\in\I_N^d}u_k \underbrace{\sum\limits_{j\in{\I'}_{NS}^d} |j|^{2s} \e^{\i \frac{2\pi j (\kappa-k)}{SN}}}_{=\tilde\Phi(\kappa-k)},
	\]
	see \ref{eq:expl_form_scaled_periodic}.
Using \ref{eq:PHI_1} one can verify that it is enough to show that
\[
	\Phi^{N}(\kappa - k) = \frac{(2\pi)^{2s}}{S^{2s}{NS}^d}\tilde\Phi(\kappa-k)
\]
for an appropriate choice of $(\alpha_i, x_i)_{i=1,\ldots,N_Q}$. From
	\ref{eq:quadrature_rules_phi_hat}, we derive using \ref{eq:naive_quadrature}
	that
	\begin{equation*}
		\hat\Phi_{k} = C(N,d,s) \frac{1}{N_Q^d}
		\sum\limits_{j\in{\I'_{2N N_Q}}}\left| \frac{j}{N_Q} \right|^{2s}	
			Y_d\left(\frac{\pi}{N} \left( \frac{j}{N_Q} -k\right)\right).
	\end{equation*}
We then calculate the DFT of $\tilde\Phi$ and obtain

\begin{align*}
	\hat{\tilde\Phi}_k &=\frac{(2\pi)^{2s}}{S^{2s}NS} \sum\limits_{\kappa\in{\I'}_N^d} \tilde\Phi_\kappa \e^{-\i \frac{2\pi\kappa k}{2N}}\\
		&=\frac{(2\pi)^{2s}}{S^{2s}{NS}^d}\sum\limits_{\kappa\in{\I'}_N^d} \left(\sum\limits_{j\in{\I'}_{NS}^d} |j|^{2s} \e^{\i \frac{2\pi j \kappa}{SN}} \right) \e^{-\i \frac{2\pi\kappa k}{2N}} \\
		&=\frac{(2\pi)^{2s}}{S^{2s}{NS}^d} \sum\limits_{j\in{\I'}_{NS}^d} |j|^{2s}  
			Y\left(2\pi \left( \frac{j}{SN} - \frac{k}{2N}   \right)\right) \\
		&=\frac{\pi^{2s}}{(2N)^d} \frac{1}{{N_Q}^d} \sum\limits_{j\in{\I'}_{2N N_Q}^d} \left| \frac{j}{N_Q}  \right|^{2s}  
			Y\left( \frac{\pi}{N} \left( \frac{j}{N_Q} - k \right)\right),
\end{align*}

which completes the proof.	
\end{proof}

\subsection{Solving the Dirichlet problem}
\label{sub:solving_the_system}
In \cref{sec:interpolation}, we have seen how to implement the
application of the discrete operator $\Phi^N$ to a vector $\u\in\R^{N^d}$
efficiently using the fast Fourier transformation algorithm. In this section, we
will show how this is used to solve the \textit{fractional Poisson problem} with Dirichlet exterior conditions, i.e., \cref{eq:bvp}, repeated here for convenience:
\begin{equation}
	\text{find }u\text{ s.t. }
	\begin{cases}
		\flap u = f &\text{in } \Omega \\
		u = 0 &\text{in }\R^d\setminus\Omega
	\end{cases}
\end{equation}
where $\Omega \subset [0;1)^d\subset \R^d$ is an arbitrary Lipschitz
domain and $f$ is a given function. 
Originally, our methods operates on the full cube $[0;1)^d$. This leads to the
discretized problem
\begin{equation}
	\label{eq:discrete_system}
	\text{find } \u \in \R^{N^d} \text{ s.t. }\Phi^N \u = \f
\end{equation}
where $\f = (f_k)_{k \in \I_N^d}$, $f_k = f(x_k)$, $x_k = k/N$ and $\Phi^N$ is the
discrete operator from \cref{eq:PHI_1}. As shown,
the application of $\Phi^N$ to a vector $\u\in\R^{N^d}$ can be implemented
efficiently using the fast Fourier transformation algorithm. It is thus feasible
to solve \cref{eq:discrete_system} using iterative methods that work through
subsequent applications of the operator instead of inverting them directly. In
the present case, we use the conjugate gradient method \cite{Hestenes1952}
as it is fast, easy to implement and numerically stable. The same procedure has been used by Duo and Zhang in \cite{Duo2018} to solve the fractional Poisson problem in two-and three dimensions on rectangular (or cuboid) domains using their finite difference method,
and by Minden and Ying \cite{Minden2020} to solve the discrete system they obtain using singularity subtraction. To
overcome the issue of being restricted to the cube $\Omega = [0;1)^d$ and solve 
problems on arbitrary domains $\Omega \subset [0;1)^d$, we embed the domain
into the cube and set the coefficients outside of $\Omega$ to $0$.
To implement this, we introduce a linear mapping
\[
	S_\Omega : \R^{N^d}\longrightarrow\R^{N^d}
\]
such that for all $\u = (u_k)_{k\in\I_N^d} \in\R^{N^d}$
\[
	(S_\Omega \u)_k = \begin{cases}
		u_k&\text{if }k/N\in\Omega\\
		0&\text{otherwise}
	\end{cases}\
\]
holds. Additionally, we define $S_D \coloneqq \id - S_\Omega$ where $\id$ is
the identity on $\R^{N^d}$. Now, as we only want to solve
\cref{eq:discrete_system} for the indices that belong to $\Omega$, we solve 
the modified problem
\begin{equation}
	\label{eq:linear_system}
	\text{find }\u\in\R^{N^d}\text{ s.t. }
	\begin{cases}
		S_\Omega\Phi^N S_\Omega^T\u  = S_\Omega\f \\
		S_D u = 0
	\end{cases}
\end{equation}
instead. Note that this is not a system on $\R^{N^d}$ anymore, but
only on the subspace spanned by the indices that are selected by $S_\Omega$.
The calculations are still done on the full $\R^{N^d}$ as this is the space where
we can apply $\Phi^N$ efficiently, which is the prerequisite for using the
conjugate gradient method.  To show the capabilities of the methods, we apply
it benchmark problems and problems arising from applications.
\section{Numerical Experiments}
\label{sec:numexp}
We present four numerical examples to demonstrate the efficiency of our implementation. First, we compare the scaled periodic fractional Laplacian (see \cref{sub:calculation_periodic_flap}) to the sinc-fractional Laplacian for different scaling factors. \Cref{sub:const_rhs} provides an experimental error analysis for a case where analytic solutions of the Dirichlet problem are known explicitly.
As an example for the importance of using the correct exterior value conditions, we compare the fractional Allen-Cahn evolution equation for the periodically extendend and the Dirichlet case in \cref{sub:allen_cahn}. Finally, in \cref{sub:image_denoising}, we show an application to image denoising as introduced in \cite{AntilBartels2017}.
\subsection{Quadrature rules for the convolution kernel}
\label{sub:experiments_integration_phi_hat}
In this section, we want to numerically evaluate how different quadrature rules
in the calculation of the convolution kernel behave. For that reason, we calculate
the discrete scaled periodic Fractional Laplacian $\flapp_{S,N}$ of a function $u$ with
different scaling factors $S$ and compare it to the sinc-fractional
Laplacian $\flap_{N}$ of $u$ computed using different quadrature rules in the integration in
\cref{eq:quadrature_rules_phi_hat}.

We obtain the simple error estimate
\begin{equation}
	\|\flap u - \flap_N u\| \leq \|\flap u - \flapp_{S,N} u\| + \|\flapp_{S,N}u - \flap_{N} u\|
	\label{eq:error_ext_per}
\end{equation}
From \cref{thm:diff_flappS_flap}, we know that the first term is $\mathcal{O}(S^{-d+2s})$,
and the second term can be calculated using the methods presented in
\cref{sec:discrete_operator,sub:calculation_periodic_flap,sub:experiments_integration_phi_hat}.

In order to obtain good accuracy for calculating the sinc-fractional Laplacian at a
reasonable computational cost, we aim to use better quadrature rules than the
naive one presented in \cref{eq:naive_quadrature}. Basically, we can use any quadrature rule, but for efficient implementation, we
should use the same quadrature rule on each of the $d$-dimensional cubes
$[j_1;j_1+1]\times\cdots\times[j_d;j_d+1]$. This causes some issues as some cubes 
contain an integrand with a singularity due to the factor $|j +
\omega|^{2s}$ if $j_i \in \{-1,0\}$ for some $i \in \{1,\ldots,d\}$.
Nevertheless, we obtain good results using a Gauß-Legendre quadrature rule.
The results are illustrated in \cref{fig:comp_shannon_peri}, where we show the second term in the estimate \cref{eq:error_ext_per}. Due to \cref{thm:equiv_scaledperi_shannon}, a sinc-fractional Laplacian with exactly computed convolution kernel corresponds to a periodic fractional Laplacian with `infinite' scaling. In that case the difference between the sinc-fractional Laplacian and $\flapp_{S,N}$ would decrease at the rate $\mathcal O(S^{-d+2s})$. In the figure, one can see that this rate applies until the quadrature errors from the computation of the convolution kernel dominate. Naturally, this happens later for more exact quadratures.  In the numerical experiments below, we thus employ an $n^d$-point tensor product  Gauß-Legendre quadrature, with $n=7$ for $d=2$ and $n=5$ for $d=3$. \cref{fig:comp_shannon_peri} shows that a 7-point Gauß-Legendre quadrature corresponds roughly to a scale factor $S$ between 40 and 60, depending on $s$.

\begin{figure}
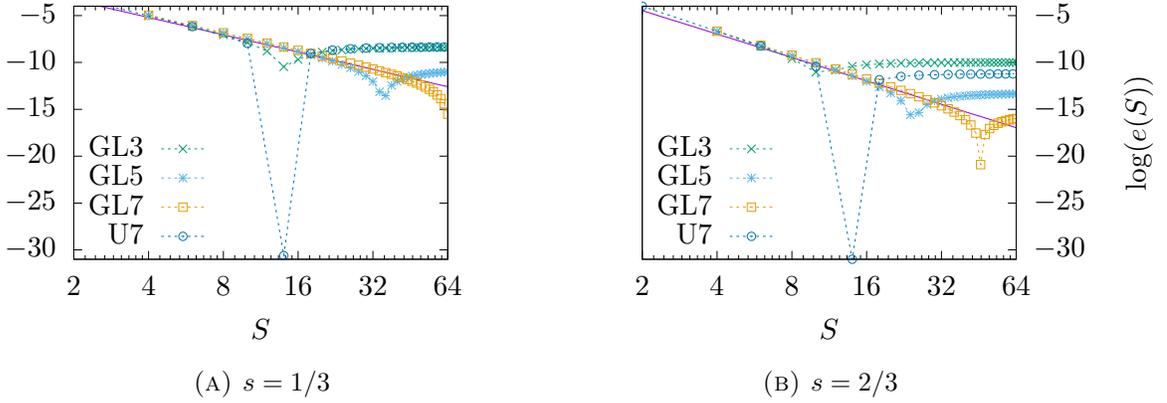

	\begin{subfigure}{.45\textwidth}
	    \centering
			\input{paper-sinc-fraclap-gnuplottex-fig1_tex.tex}
			\caption{$s = 1/3$}
		\end{subfigure}
		\begin{subfigure}{.45\textwidth}
		    \centering
				\input{paper-sinc-fraclap-gnuplottex-fig2_tex.tex}
			\caption{$s = 2/3$}
	\end{subfigure}
	\caption{Comparison of the scaled periodic fractional Laplacian and our implementation 
	of the sinc-fractional Laplacian applied to a standard mollifier with support in $[0;1)^2$ for fixed $N=64$. The difference is measured as 
	$e(S) = \left\|\flap_N u - \flapp_{N,S} u\right\|_{\ell^\infty}$.
	The solid purple line shows the rate $\mathcal O(S^{-d+2s})$. GL3, GL5 and GL7 denote Gauss-Legendre integration with 3, 5 and 7 points in each spatial direction,
    U7 denotes $7$ uniformly spaced quadrature points in each spatial direction. We 
	see that at 
	$S = 2N_Q = 14$, the difference $e(S)$ is practically zero for this quadratue rule.
	This illustrates \cref{thm:equiv_scaledperi_shannon_2}.
	}
	 
	\label{fig:comp_shannon_peri}
\end{figure}

\subsection{Function with constant fractional Laplacian on the unit sphere} \label{sub:const_rhs}
One of the few examples where the solution to the fractional Laplace Dirichlet problem is known explicitly is
the problem
\begin{equation}
	\label{eq:frac_poiss_const_rhs}
	\text{find }u\text{ s.t. }
	\begin{cases}
		\flap u = 1&\text{in }\Omega \\
		u = 0&\text{in }\R^d\setminus\Omega.
	\end{cases}
\end{equation}
For $\Omega = \{ x\in\R^d \,|\, |x| < 1 \}$, the solution to  \cref{eq:frac_poiss_const_rhs} is given by
(see \cite{ClaudiaBucur2016})
\begin{equation}
	\label{eq:sol_frac_poiss_const_rhs}
	u = C_u(d,s)  \max\{0, (1 - |x|^2)\}^s
\end{equation}
where $ C_u(d,s) = \Gamma\left( d/2\right) \cdot (2^{2s} \Gamma\left( d/2+s \right)\Gamma(1+s))^{-1}$.
After shifting and scaling the problem to a disc or a sphere that is a subset of the cube
$[0;1)^d$, and using the method from \cref{sub:solving_the_system}, we obtain the results shown in \cref{fig:first_ex_solver} in
the case $d=2$, which clearly resemble the expected results. 
In \cref{tab:experimental_num_iterations}, we show the number of iterations of the 
conjugate gradient method required  until the residual 
\(
    \|\mathbf r\|_{L^2} = \frac{1}{N^d} \sum_{k\in\I_N^d} r_k^2
\)
dropped below $10^{-8}$and in \cref{tab:experimental_num_iterations_scaling}, we
show how the number of iterations scales w.r.t $N$.
  Since the conditioning of the problem becomes worse for finer grid resolution, it is expected that more iterations are necessary for increasing $N$. We note that for lower fractional exponent $s$, the required number of iterations is lower -- this is reasonable as the largest eigenvalue of the sinc-fractional Laplacian should scale like $N^{2s}$.

\begin{table}
	\caption{Number of CG-iterations needed to solve the discretized linear system.
	The total number of Degrees of Freedom (DoFs) is $\mathcal O(N^d)$.}\label{tab:experimental_num_iterations}
	\begin{center}
		\begin{tabular}{rc|r|r|r|r|r|r}
			&$N$ & $s = 1/4$ & $s = 1/3$ & $s = 1/2$ & $s = 2/3$ & $s = 3/4$ & $s=1$ \\ \hline\hline
\multirow{10}{*}{$d=2$} &    8 &  8 &   8 &   8 &    8 &    8 &    8 \\ \cline{2-8}
												&   16 & 14 &  17 &  21 &   24 &   26 &   27 \\ \cline{2-8}
												&   32 & 19 &  24 &  34 &   45 &   51 &   63 \\ \cline{2-8}
												&   64 & 25 &  32 &  48 &   75 &   91 &  132 \\ \cline{2-8}
												&  128 & 31 &  43 &  76 &  127 &  161 &  271 \\ \cline{2-8}
												&  256 & 40 &  57 & 112 &  208 &  281 &  545 \\ \cline{2-8}
												&  512 & 50 &  75 & 163 &  340 &  488 & 1089 \\ \cline{2-8}
												& 1024 & 61 &  97 & 234 &  550 &  882 & 2056 \\ \cline{2-8}
												& 2048 & 76 & 127 & 343 & 1021 & 1660 & 4257 \\ \cline{2-8}
												& 4096 & 94 & 166 & 535 & 1668 & 3010 & 8406 \\ \hline \hline
\multirow{6}{*}{$d=3$}  &    8 & 10 &  11 &  12 &   13 &   13 &   13 \\ \cline{2-8}
												&   16 & 14 &  16 &  22 &   28 &   31 &   41 \\ \cline{2-8}
												&   32 & 18 &  23 &  33 &   48 &   57 &   92 \\ \cline{2-8}
												&   64 & 23 &  30 &  49 &   79 &  101 &  189 \\ \cline{2-8}
												&  128 & 29 &  40 &  73 &  128 &  174 &  386 \\ \cline{2-8}
												&  256 & 36 &  53 & 107 &  211 &  299 &  795 \\ \cline{2-8}
		\end{tabular}
	\end{center}
\end{table}

\begin{table}[h!]
	\caption{Scaling exponent $\beta$ of the number $N^\beta$ of CG-iterations depending on grid size $N$. }
	\label{tab:experimental_num_iterations_scaling}
	\begin{center}
		\begin{tabular}{|rr|c|c|c|c|c|c|}\hline 
				& $s$ &1/4	&	1/3 & 1/2	& 2/3 & 3/4 & 1 \\ \hline 
				 \multirow{2}{*}{$\beta$}& $d = 2$    & 0.27  & 0.36 & 0.55 & 0.71 & 0.76 & 1.02 \\ \cline{3-8}
				 & $d = 3$    &  0.32 & 0.41 & 0.56 & 0.71 & 0.78 & 1.04 \\ \hline
		\end{tabular}
	\end{center}
\end{table}
In \cref{tab:t_solve_combined}, we present the time that was required solve the system.
The time grows when $s$ grows, reflecting the fact that we need more CG-iterations in this case.
We implemented the algorithms in \verb|C++| using the FFTW library \cite{FFTW05} and the 
experiments where run on a standard office computer (6-core Intel Core i5-9500, 3.00 Ghz).
In addition to the time that is needed to actually solve the system, one has to setup
the operator $\Phi^N$. This convolution kernel has to be computed only once for each $d$, $s$ and mesh size $N$. It is independent of the domain $\Omega \subset [0,1)^d$. Especially for large values of $s$, this time is small compared
to the time that is needed to actually solve the system. We provide the details
in \cref{tab:t_setup_combined}.
\begin{table}
	\caption{Time (in seconds) needed to solve the discretized system. These are the
	actual times it took to compute the results used for \cref{fig:error_analysis_2d}.
	}\label{tab:t_solve_combined}
		\resizebox{\textwidth}{!}{%
			\begin{tabular}{cr|r|r|r|r|r|r|r}
				&$N$ & \#DoF & {$s = 1/4$} & {$s = 1/3$} & {$s = 1/2$} & {$s = 2/3$} & {$s = 3/4$} & {$s = 1$} \\ \hline\hline
			\multirow{8}{*}{$d=2$}&   32 &        651 & 1.0e$-$3 & 1.0e$-$3 & 2.0e$-$3 & 2.0e$-$3 & 3.0e$-$3 & 4.0e$-$3\\ \cline{2-9}
			                      &   64 &      2,605 & 4.0e$-$3 & 5.0e$-$3 & 7.0e$-$3 & 1.2e$-$2 & 1.3e$-$2 & 2.7e$-$2\\ \cline{2-9}
			                      &  128 &     10,423 & 1.4e$-$2 & 1.8e$-$2 & 3.2e$-$2 & 5.3e$-$2 & 7.2e$-$2 & 1.5e$-$1\\ \cline{2-9}
			                      &  256 &     41,692 & 5.1e$-$2 & 7.3e$-$2 & 1.4e$-$1 & 2.8e$-$1 & 3.8e$-$1 & 1.0e$+$0\\ \cline{2-9}
			                      &  512 &    166,768 & 3.8e$-$1 & 5.6e$-$1 & 1.2e$+$0 & 2.6e$+$0 & 3.8e$+$0 & 1.2e$+$1\\ \cline{2-9}
			                      & 1024 &    667,075 & 1.9e$+$0 & 3.0e$+$0 & 7.0e$+$0 & 1.7e$+$1 & 2.7e$+$1 & 1.0e$+$2\\ \cline{2-9}
			                      & 2048 &  2,668,300 & 1.1e$+$1 & 1.8e$+$1 & 4.4e$+$1 & 1.2e$+$2 & 1.9e$+$2 & 8.5e$+$2\\ \cline{2-9}
			                      & 4096 & 10,673,203 & 5.8e$+$1 & 1.0e$+$2 & 3.2e$+$2 & 9.8e$+$2 & 1.7e$+$3 & 8.7e$+$3\\ \hline \hline
	      \multirow{5}{*}{$d=3$}  &   16 &      1,563 & 7.0e$-$3 & 8.0e$-$3 & 1.0$-$2 & 1.3e$-$2 & 1.5$-$2 & 2.0e$-$2\\ \cline{2-9}
                                  &   32 &     12,507 & 9.9e$-$2 & 1.3e$-$1 & 1.8$-$1 & 2.6e$-$1 & 3.1$-$1 & 5.3e$-$1 \\ \cline{2-9}
                                  &   64 &    100,061 & 1.4e$+$0 & 1.8e$+$0 & 3.0$+$0 & 4.8e$+$0 & 6.0$+$0 & 1.1e$+$1 \\ \cline{2-9}
                                  &  128 &    800,490 & 1.9e$+$1 & 2.6e$+$1 & 4.8$+$1 & 8.1e$+$1 & 1.1$+$2 & 2.5e$+$2 \\ \cline{2-9}
                                  &  256 &  6,403,922 & 2.0e$+$2 & 3.0e$+$2 & 6.9$+$2 & 1.3e$+$3 & 2.1$+$3 & 5.5e$+$3 \\ \cline{2-9}
		\end{tabular}}
\end{table}

\begin{table}
	\caption{Time (in seconds) needed to setup $\Phi^N$ in $d = 2$ and $d=3$ dimensions using a $7^2$ ($d=2$) or $5^3$ ($d=3$) point Gauß-Legendre quadrature. Notice that $\Phi^N$ is applied to a vector of size $N^d$ and that $\Phi^N$ has 
	$(2N)^d$ entries}.\label{tab:t_setup_combined}
		\resizebox{.95\textwidth}{!}{%
			\begin{tabular}{lc|r|r|r|r|r|r}
				&{$N$} & {$s = 1/4$} & {$s = 1/3$} & {$s = 1/2$} & {$s = 2/3$} & {$s = 3/4$} & {$s = 1$} \\ \hline\hline
\multirow{10}{*}{$d=2$}  &    8 & 1.4e$-$2 & 1.5e$-$2 & 1.4e$-$2 & 1.5e$-$2 & 1.5e$-$2 & 1.4e$-$2  \\ \cline{2-8}
											 &   16 & 5.2e$-$2 & 5.1e$-$2 & 5.1e$-$2 & 5.2e$-$2 & 5.2e$-$2 & 5.0e$-$2  \\ \cline{2-8}
											 &   32 & 1.1e$-$1 & 1.1e$-$1 & 9.7e$-$2 & 1.1e$-$1 & 1.1e$-$1 & 9.8e$-$2  \\ \cline{2-8}
											 &   64 & 3.1e$-$1 & 3.0e$-$1 & 2.7e$-$1 & 3.0e$-$1 & 3.0e$-$1 & 2.7e$-$1  \\ \cline{2-8}
											 &  128 & 9.5e$-$1 & 9.5e$-$1 & 8.2e$-$1 & 9.5e$-$1 & 9.5e$-$1 & 8.2e$-$1  \\ \cline{2-8}
											 &  256 & 3.4e$+$0 & 3.4e$+$0 & 2.9e$+$0 & 3.4e$+$0 & 3.4e$+$0 & 2.9e$+$0  \\ \cline{2-8}
											 &  512 & 1.3e$+$1 & 1.3e$+$1 & 1.1e$+$1 & 1.3e$+$1 & 1.3e$+$1 & 1.1e$+$1  \\ \cline{2-8}
											 & 1024 & 5.2e$+$1 & 5.2e$+$1 & 4.3e$+$1 & 5.2e$+$1 & 5.2e$+$1 & 4.3e$+$1  \\ \cline{2-8}
											 & 2048 & 2.1e$+$2 & 2.1e$+$2 & 1.7e$+$2 & 2.1e$+$2 & 2.1e$+$2 & 1.7e$+$2  \\ \cline{2-8}
											 & 4096 & 8.2e$+$2 & 8.2e$+$2 & 6.9e$+$2 & 8.2e$+$2 & 8.2e$+$2 & 6.9e$+$2  \\ \hline\hline
\multirow{6}{*}{$d=3$}   &    8 & 1.8e$-$1 & 1.8e$-$1 & 1.5e$-$1 & 1.8e$-$1 & 1.8e$-$1 & 1.6e$-$1  \\ \cline{2-8}
			                 &   16 & 1.4e$+$0 & 1.4e$+$0 & 1.2e$+$0 & 1.4e$+$0 & 1.4e$+$0 & 1.2e$+$0  \\ \cline{2-8}
			                 &   32 & 1.1e$+$1 & 1.1e$+$1 & 9.4e$+$0 & 1.1e$+$1 & 1.1e$+$1 & 9.4e$+$0  \\ \cline{2-8}
			                 &   64 & 8.7e$+$1 & 8.7e$+$1 & 7.6e$+$1 & 8.6e$+$1 & 8.7e$+$1 & 7.6e$+$1  \\ \cline{2-8}
			                 &  128 & 7.1e$+$2 & 7.1e$+$2 & 6.2e$+$2 & 7.1e$+$2 & 7.1e$+$2 & 6.2e$+$2  \\ \cline{2-8}
			                 &  256 & 5.7e$+$3 & 5.7e$+$3 & 5.0e$+$3 & 5.7e$+$3 & 5.7e$+$3 & 5.0e$+$3  \\ \cline{2-8}
		\end{tabular}}
\end{table}

\Cref{fig:first_ex_solver3d} shows the results for $d=3$. In this case, we used only a $5$-point Gauss-Legendre rule in order to reduce computation time. 

\begin{figure}
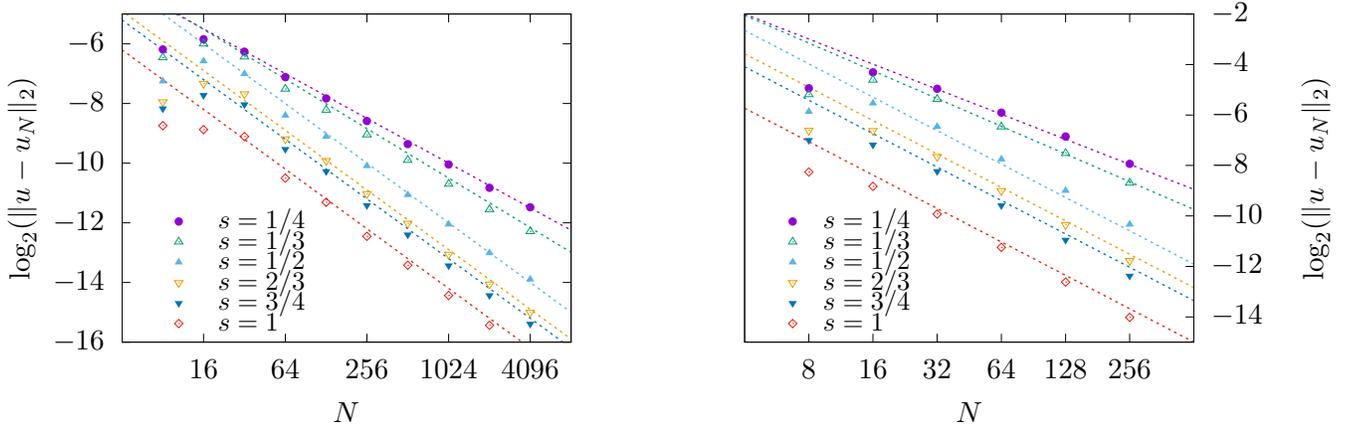

    \centering
		\begin{subfigure}{.5\textwidth}
		    \centering
				\input{paper-sinc-fraclap-gnuplottex-fig3_tex.tex}
		\end{subfigure}%
		\begin{subfigure}{.5\textwidth}
		    \centering
				\input{paper-sinc-fraclap-gnuplottex-fig4_tex.tex}
		\end{subfigure}
	\caption{Experimental convergence analysis as a log-log plot in 2 dimensions (left) and 3 dimensions (right). 
	The decay conforms to the rates predicted  in \cite{Acosta2017, Borthagaray2018, Ainsworth2018}, displayed as dashed lines in the plots. $N$ is the number of
	grid points in each direction, i.e., the total number of grid points is $N^2$
	and $N^3$ in respectively.}
	\label{fig:error_analysis_2d}
\end{figure}

As a numerical analysis of our method is still pending and remains part of future
work, we experimentally evaluate the capabilities of our method. For that, we
solve problem \cref{eq:frac_poiss_const_rhs} for different values of $s$ on
grids of increasing size $N^d$ for $d = 2$ and $d = 3$. We approximate the $L_2$-error as
\[
	\left\|u_N - u\right\|_2 \approx \sqrt{\frac{1}{N^d} \sum\limits_k (u_k - u(x_k))^2},
\]
where $u_N = \sum_{k\in \I_N^d} u_k \varphi^N_k$ is the solution computed using the sinc-fractional Laplacian and $u$ is the known analytic solution.
The results can be seen in \cref{fig:error_analysis_2d} (left) for the $2d$-case
and in \cref{fig:error_analysis_2d} (right) for the $3d$-case.
We experimentally obtain the convergence rates shown in
\cref{tab:experimental_gamma}. 
In \cite{Acosta2017}, Acosta and Borthagaray proved for their finite element implementation
the convergence rate $\mathcal O(h^{1-\eps})$ for mesh size $h=1/N$ in the $H^s(\Omega)$-norm
under appropriate smoothness assumptions on the domain $\Omega$. Using an Aubin-Nitsche
argument, the convergence in the $L^2(\Omega)$-norm is $\mathcal{O}\left(h^{\min(1, s+1/2)}\right)$
\cite{Borthagaray2018, Ainsworth2018}, modulo $\eps$ or a logarithmic correction.
The rates that we obtain clearly recover this rate for $s<1$.  Our method can also treat the case $s=1$, i.e., the standard Laplacian. Notice, however, the reduced rate in \cref{tab:experimental_gamma}. This is due to the fact that the exact solution $u \notin H^2(\R^d)$.

\begin{table}[h!]
	\caption{Experimentally determined convergence rates.}
	\label{tab:experimental_gamma}
	\begin{center}
		\begin{tabular}{rr|c|c|c|c|c|c}\hline 
			 & $s$&1/4	&	1/3 & 1/2	& 2/3 & 3/4 & 1\\ \hline 
				\multirow{2}{*}{Determined rate} & $d = 2$  & 0.7329 & 0.8192 & 0.9622 & 1.0166 & 1.0189 & 1.0126  \\ \cline{3-8} 
																				 & $d = 3$  & 0.7439 & 0.8324 & 0.9725 & 1.0360 & 1.0425 & 1.0306 \\ \hline
				Expected rate   & & 0.75 & $0.8\bar 3$ & 1.00 & 1.00 &  1.00 & 1.00\\ \hline
		\end{tabular}
	\end{center}
\end{table}

\begin{figure}[h!]
	\begin{subfigure}{.48\textwidth}
		\includegraphics[width=\textwidth]{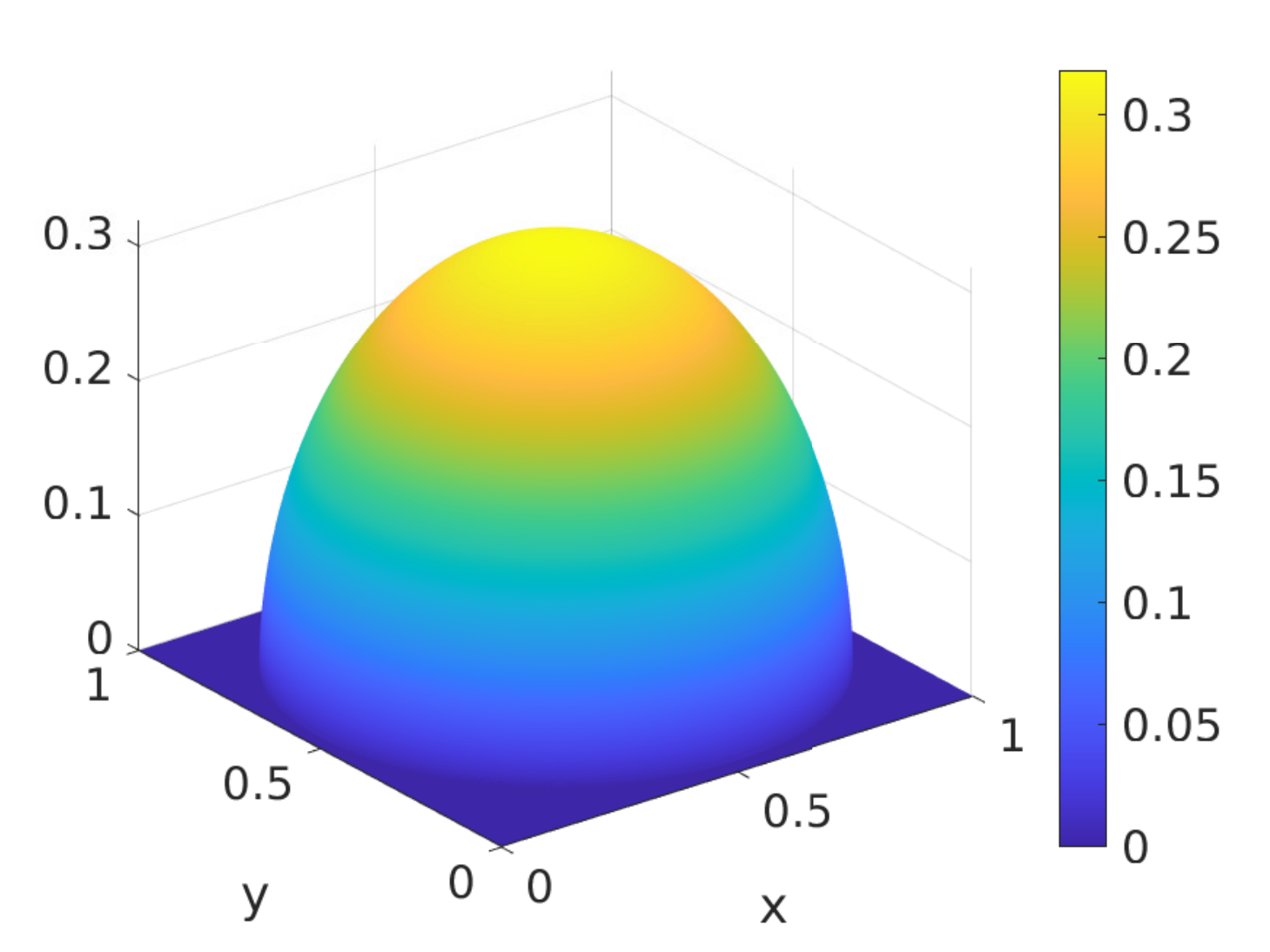}
		\caption{$\u$}
	\end{subfigure}
	\hfill
	\begin{subfigure}{.48\textwidth}
		\includegraphics[width=\textwidth]{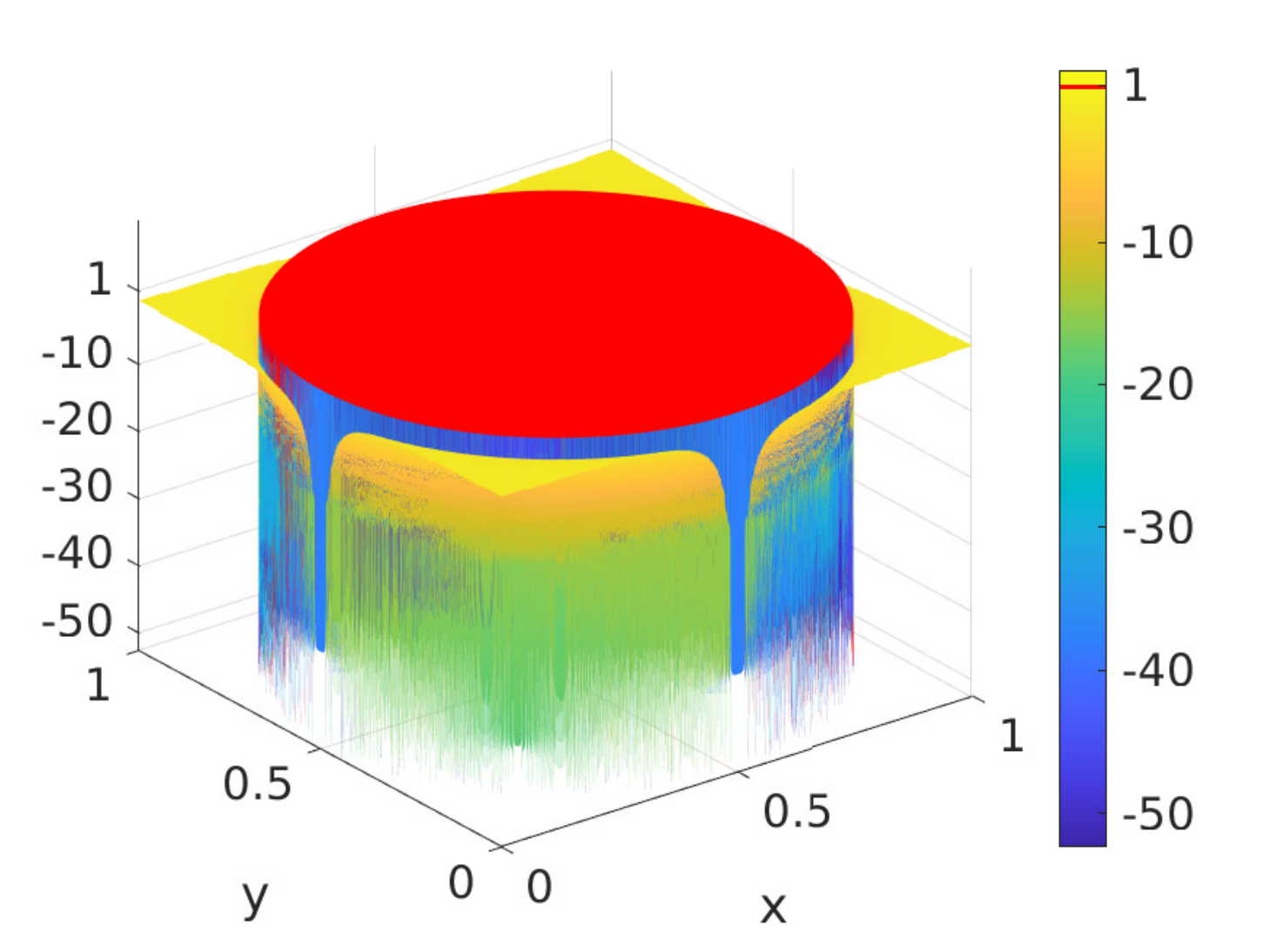}
		\caption{$\Phi\u$}
	\end{subfigure}
	\caption{Discrete solution $\u$ (a) and $\Phi\u$ (b) for problem
	\cref{eq:frac_poiss_const_rhs} with $s = 1/2$ using $2^{12}\times 2^{12}$ grid
	points. The solution clearly resembles the expected solution, given in
	\cref{eq:sol_frac_poiss_const_rhs}. In (b), values closer than $\varepsilon = 10^{-5}$ to $1$ are colored red to show that the sinc-fractional Laplacian is constant in the correct region.} \label{fig:first_ex_solver}
\end{figure}

\begin{figure}[h!]
	\begin{subfigure}{.48\textwidth}
		\includegraphics[width=\textwidth]{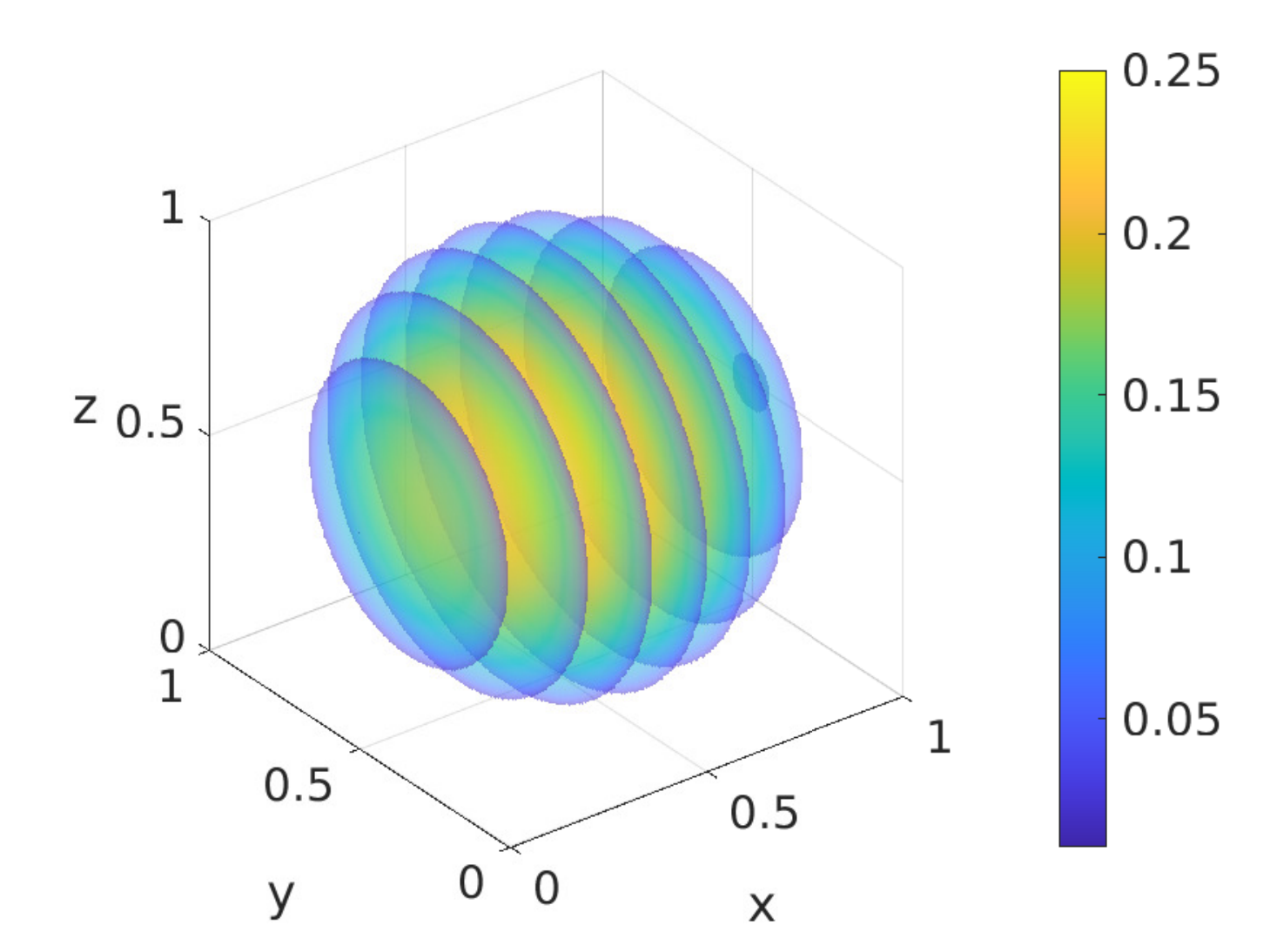}
		\caption{$\u$}
	\end{subfigure}
	\hfill
	\begin{subfigure}{.48\textwidth}
		\includegraphics[width=\textwidth]{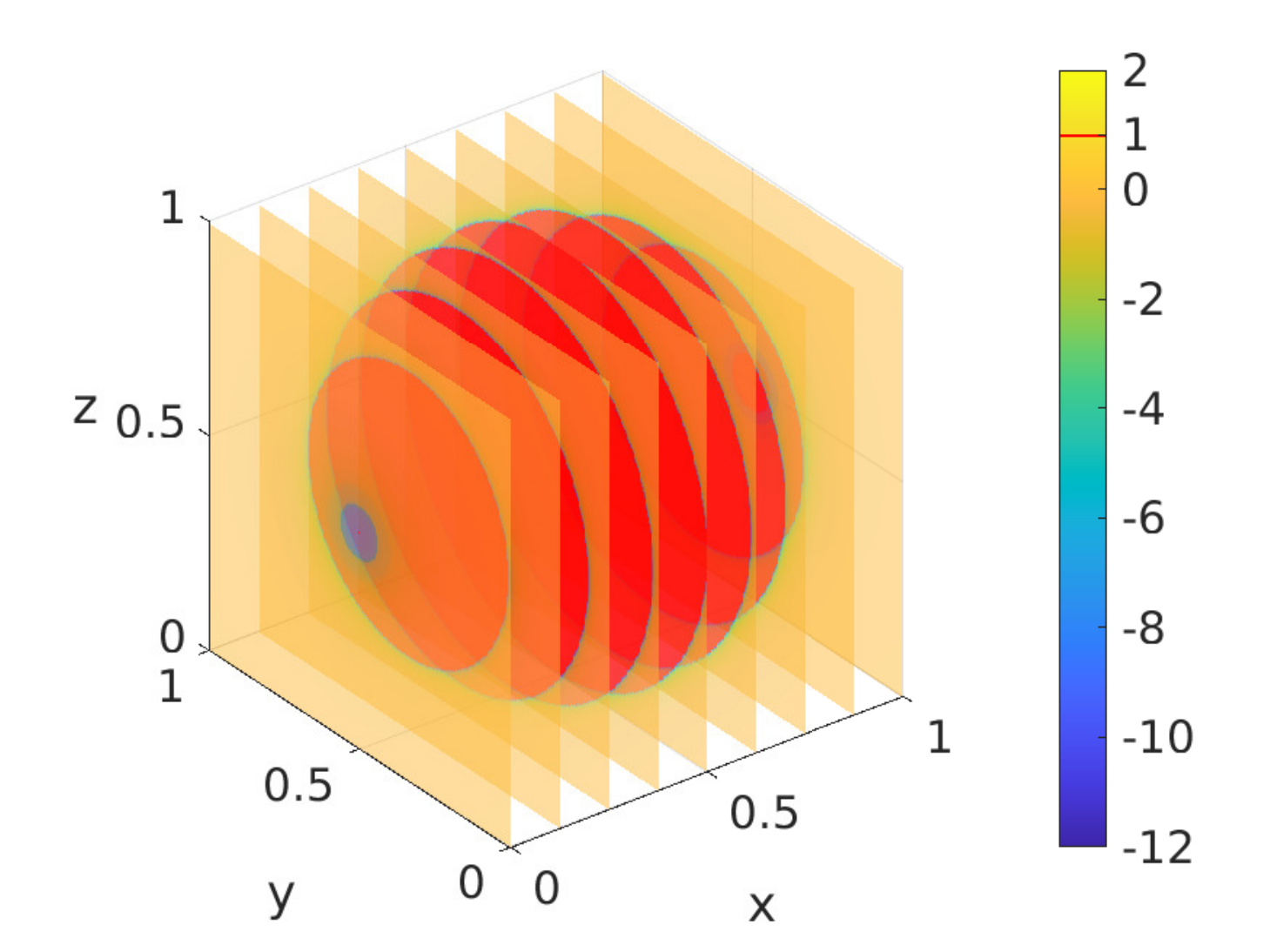}
		\caption{$\Phi\u$}
	\end{subfigure}
	\caption{Discrete solution $\u$ (a) and $\Phi\u$ (b) for problem
	\cref{eq:frac_poiss_const_rhs} with $s = 1/2$ using $(2^{8})^3$ grid points. In (b), values closer than $\varepsilon = 10^{-5}$ to $1$ are colored red to show that the sinc-fractional Laplacian is constant in the correct region.
	}
	\label{fig:first_ex_solver3d}
\end{figure}

\subsection{Function with constant fractional Laplacian on an L-shaped domain}
To show that we can  treat domains other than the sphere, we solve
the boundary value problem with constant right-hand side on an L-shaped
domain. The results are provided in \cref{fig:bvp_lshaped} and visibly
resemble the numerical solutions provided e.g. in \cite{Lischke2020}. \\
If $\Omega$ is not the unit sphere anymore, then the solution to \cref{eq:frac_poiss_const_rhs} is 
not available analytically. 
Therefore, we perform a numerical error analysis where we compare the solution on a fine mesh with $N = 2^{12}$ points in each spatial direction to solutions on coarser meshes. The results of our computations can bee seen in \cref{fig:bvp_lshaped}. We present the approximated $L^2$-errors in \cref{fig:error_analysis_2d_Lshape}. 

\begin{figure}[h!]
	\begin{subfigure}{.48\textwidth}
		\includegraphics[width=\textwidth]{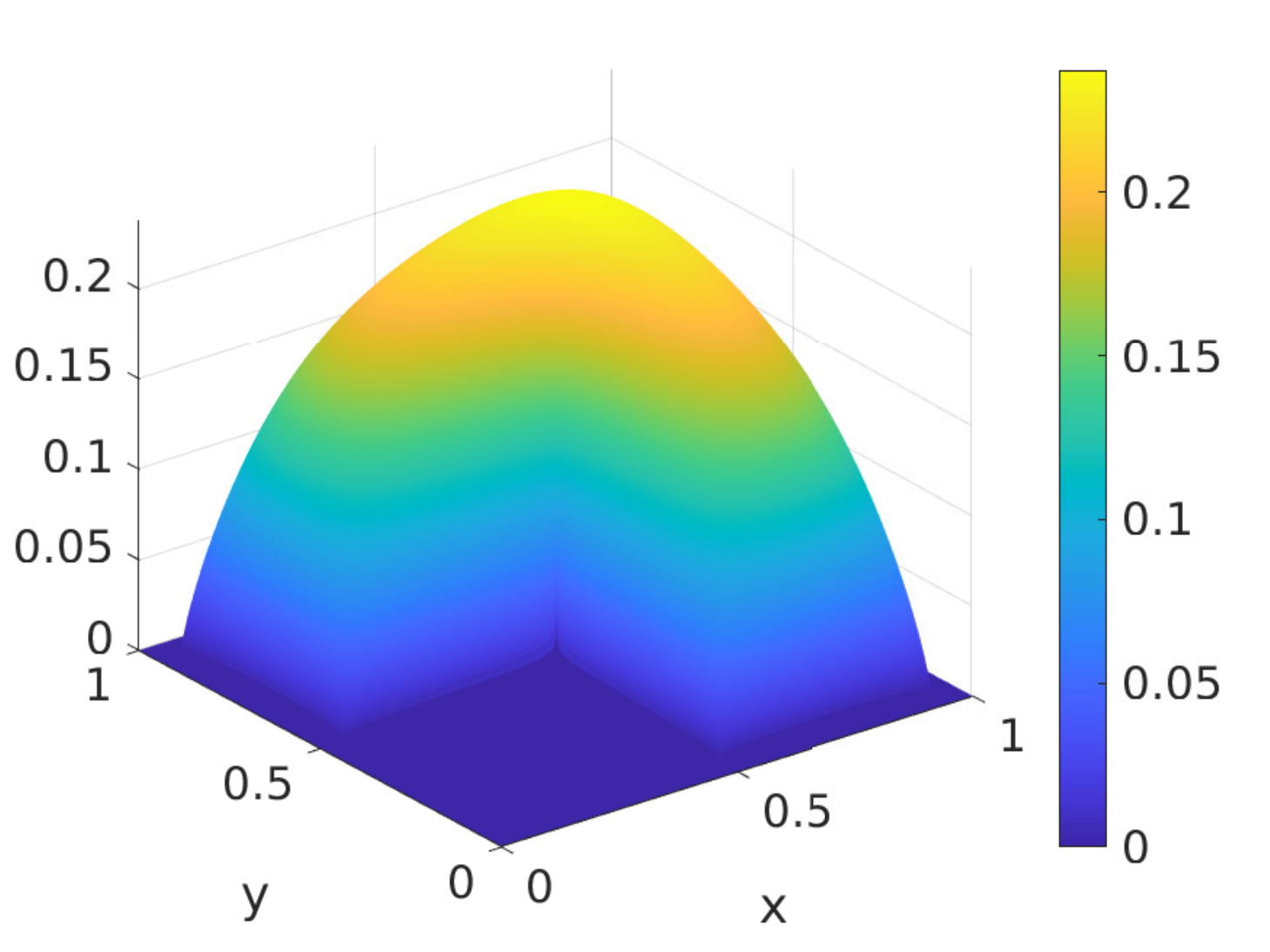}
		\caption{$\u$}
	\end{subfigure}
	\hfill
	\begin{subfigure}{.48\textwidth}
		\includegraphics[width=\textwidth]{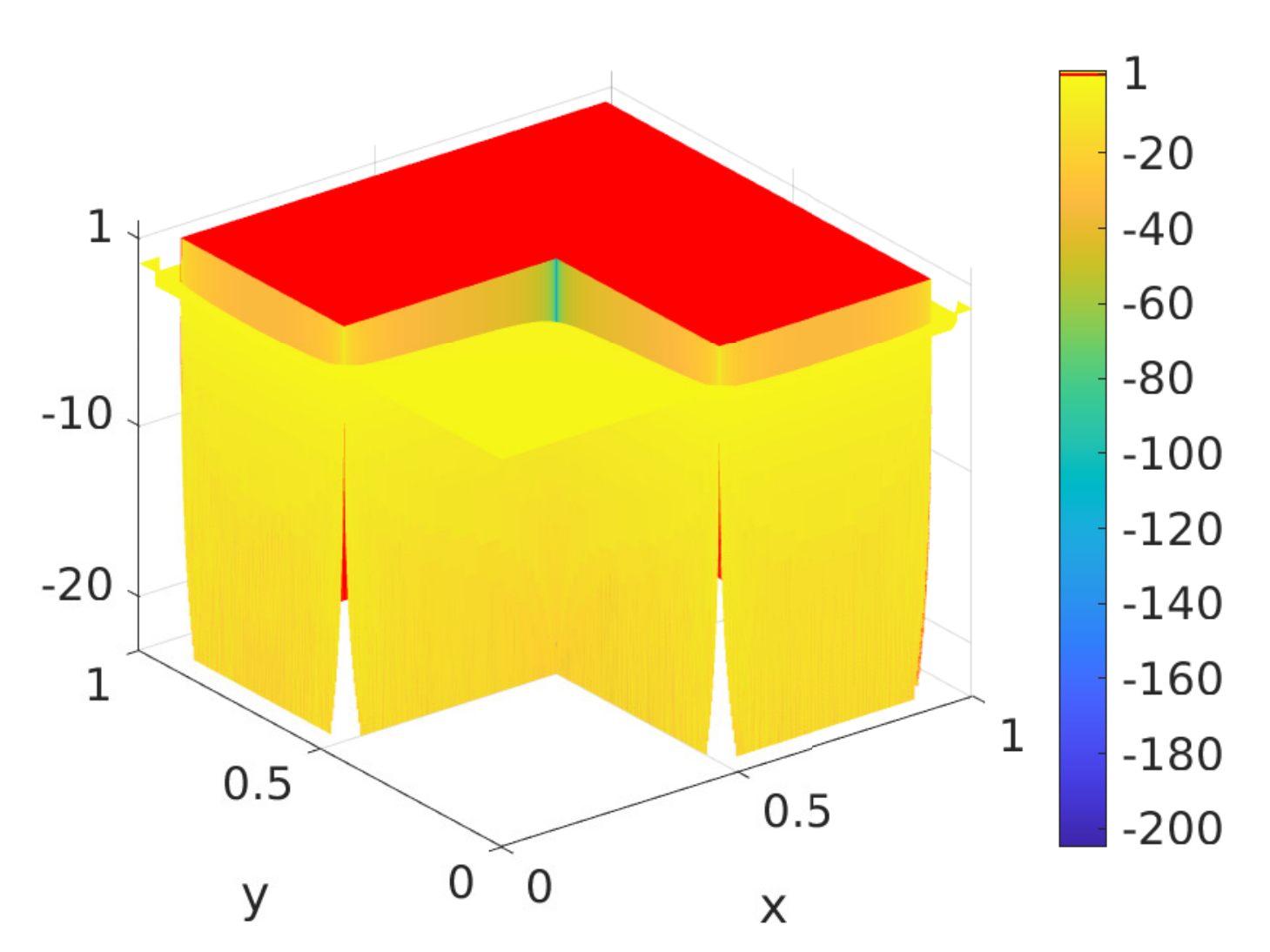}
		\caption{$\Phi\u$}
	\end{subfigure}
	\caption{Discrete solution $\u$ (a) and $\Phi\u$ (b) for problem
	\cref{eq:frac_poiss_const_rhs} with $s = 1/2$ on an L-shaped domain using $2^{12}\times 2^{12}$ grid points . In (b), values closer than $\varepsilon = 10^{-5}$ to $1$ are colored red to show that the sinc-fractional Laplacian is constant in the correct region. The fractional Laplacian exhibits a strong singularity at the inside corner.}
	\label{fig:bvp_lshaped}
\end{figure}
\begin{figure}
\centering
			\input{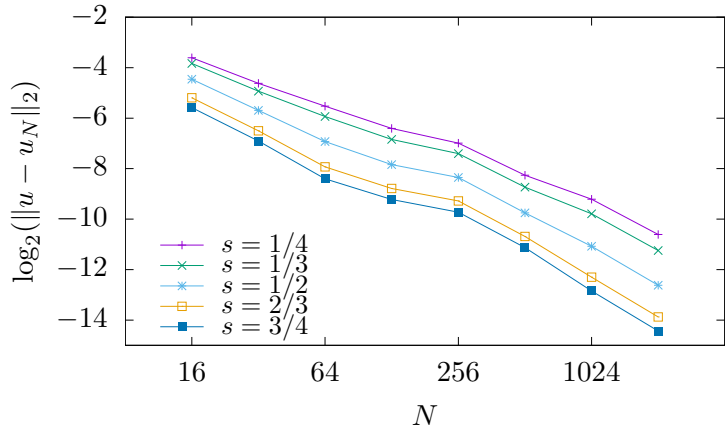}
	\caption{Experimental convergence analysis for the boundary value problem on an
	L-shaped domain. The solution on the coarse grids ($N = 2^2,\cdots,2^{11}$ in each spatial direction) where compared to the solution on
	a the finest grid ($N = 2^{12}$ in each spatial direction).}
	\label{fig:error_analysis_2d_Lshape}
\end{figure}
 
\subsection{Fractional Allen-Cahn equation} \label{sub:allen_cahn}
As a practical application that shows that our method correctly implements the
Dirichlet exterior value conditions instead of periodic exterior value conditions,
we calculate the evolution of the fractional Allen-Cahn equation
\[
	\partial_t u + (-\Delta)^{\frac{1}{2}} u  = - \frac{1}{\eps}  W'(u)
\]
for $\eps=2\cdot 10^{-3}$; as an example here for fractional exponent $s=\frac{1}{2}$. The function $W\colon \R\to\R$ is a typical quartic double well potential of the form $W(u)=\frac{1}{4}u^2(u-1)^2$. 
It has recently been proved \cite{Alberti:1994ui, Savin:2012fl} that, for $\eps\to 0$, the associated energy $\frac{1}{\log \eps} \left([u]_{H^{\frac{1}{2}}}^2 + \frac{1}{\eps}W(u)\right)$ converges (modulo constants) in the sense of $\Gamma$-convergence to the perimeter (in our one-dimensional case, a jump set counting functional). Heuristically, this energy prefers states of $u\in\{0,1\}$. The fractional Sobolev norm ensures that transitions between these two states can not take place arbitrarily rapidly in space.

The gradient flow, accelerated by a factor of $\frac{1}{\eps\log\eps}$ as computed here, converges in one spatial dimension to a kink-antikink annihilation-type dynamic \cite{Gonzalez:2012wp}. Again, heuristically, two nearby states close to $u=+1$, separated by a gap where $u=0$, attract each other due to the long range interaction via the fractional operator, so the two phase transitions (or kinks) move closer to each other. The expected behavior in the $\eps\to 0$ limit is a kink velocity proportional to the reciprocal of the distance to the antikink (and vice-versa).

Note that the evolution for small $\eps>0$ is substantially faster than for the classical local Allen-Cahn equation, where exponentially slow kink-antikink annihilation was shown \cite{BRONSARD:1990is,doi:10.1002/cpa.3160420502}.

We discretize the equation in time using an implicit Euler scheme with time step $\tau$ and obtain 
\begin{equation}
	\label{eq:sys_frac_ac}
	\left(\mathbf 1 + \tau(-\Delta)^{\frac{1}{2}}\right)\u^{t+1} = - \frac{\tau}{\eps} W'\left(\u^t\right) + \u^t.
\end{equation}

To illustrate the differences due to exterior domain condition, we choose either the periodic fractional
Laplacian or the sinc-fractional Laplacian and compare. In the first case, the system
\cref{eq:sys_frac_ac} can be solved directly using the discrete Fourier
transformation, see, e.g., \cite{AntilBartels2017} for details. In the case of the sinc-fractional Laplacian, we use the conjugate gradient method as explained in \cref{sub:solving_the_system}. We choose the domain $\Omega = [0;1) \subset \R$ and
\(
	\u^0(x) = \chi_{[1/4; 3/4]}(x),
\)
i.e., the indicator function on the interval $\left[\frac{1}{4}; \frac{3}{4} \right]$. In \cref{fig:data_evol_frac_ac_1d}, we show the evolution for different times. It can be clearly seen that the solution for periodic conditions reaches a steady state, due to symmetry: the attraction of the kink and antikink at $x=\frac{1}{4}$ and $x=\frac{3}{4} $, respectively is balanced by the attraction to their periodic mirror images past the domain boundary. In the case of the true Dirichlet problem, the annihilation is clearly visible. \Cref{fig:mass_evol_frac_ac_1d} shows the time-evolution of the total mass (i.e., $m(t)=\int_0^1 u(x,t)\,\mathrm{d}x)$ as well as the position of the left kink (or phase transition) over time. The fit to the solution of the aforementioned limit equation (which is of the form $a\sqrt{t_0-t}$ with parameters $a$ and $t_0$) is included on the left panel in \cref{fig:mass_evol_frac_ac_1d}.

\begin{figure}[h!]
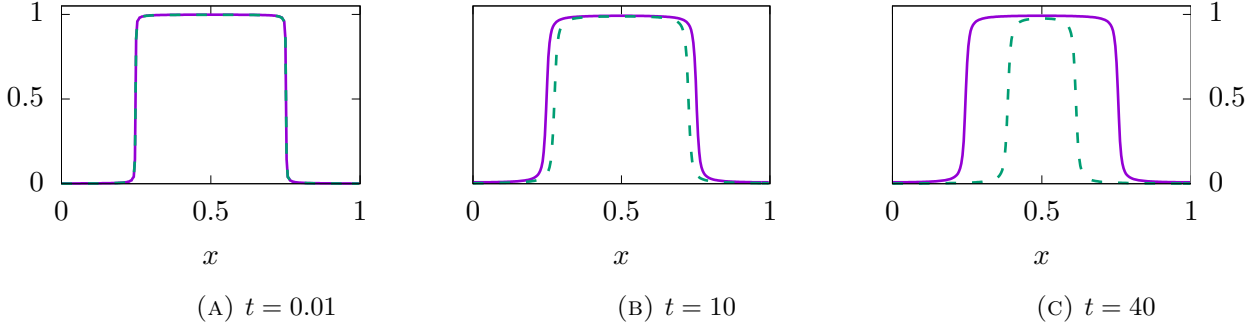

	\begin{center}
	\begin{subfigure}{.33\textwidth}
			\input{paper-sinc-fraclap-gnuplottex-fig6_tex.tex}
		\caption{$t = 0.01$}
	\end{subfigure}%
	\begin{subfigure}{.33\textwidth}
		\input{paper-sinc-fraclap-gnuplottex-fig7_tex.tex}
		\caption{$t = 10$}
	\end{subfigure}
	\begin{subfigure}{.33\textwidth}
		\input{paper-sinc-fraclap-gnuplottex-fig8_tex.tex}
		\caption{$t = 40$}
	\end{subfigure}
	\caption{The solution $u(x, t)$ of the fractional Allen-Cahn equation for different times $t$.
		The dashed geen line shows the evolution for the sinc-fractional Laplacian
		and the solid purple line the evolution for the periodic fractional Laplacian.}
		\label{fig:data_evol_frac_ac_1d}
		
	\end{center}
\end{figure}

\begin{figure}[h!]
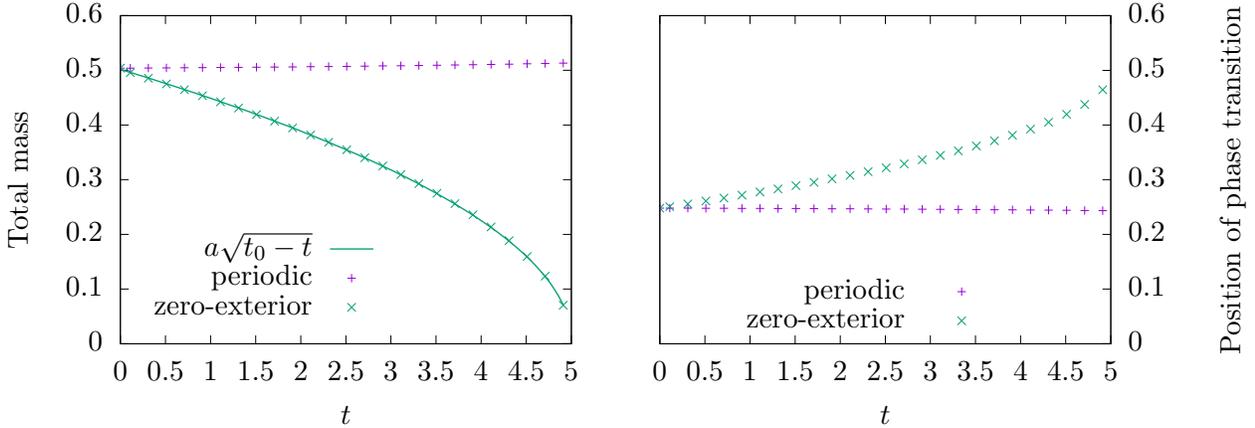

		\begin{subfigure}{.49\textwidth}
			\centering
			\input{paper-sinc-fraclap-gnuplottex-fig9_tex.tex}
		\end{subfigure}
		\begin{subfigure}{.49\textwidth}
			\input{paper-sinc-fraclap-gnuplottex-fig10_tex.tex}
		\end{subfigure}
		\caption{Evolution of the total mass (i.e., the integral of $u$) in the interval $[0;1]$
		(left) and of the position of the first phase transition (right) using both the periodic and the zero-exterior value Dirichlet (sinc)-fractional Laplacian for the fractional Allen-Cahn Equation. The continuous line in the left panel shows a fit to a square-root function $a\sqrt{t_0 - t}$, with $a = 0.2242, t_0 = 5.0104$.}
		\label{fig:mass_evol_frac_ac_1d}
\end{figure}

\subsection{Image denoising} \label{sub:image_denoising}
In \cite{AntilBartels2017}, Antil and Bartels have proposed to solve an image denoising variational problem. Given a noisy image $g$, it amounts to:
\[
	\min_{u} \frac12 \int_\Omega |(-\Delta)^{\frac{s}2}u|^2 + \frac{\alpha}{2}\int_\Omega |u-g|^2,
\]
i.e., they use the fractional Laplacian as the regularizer. Here $\alpha > 0$ is the regularization parameter. Starting from the seminal work of Rudin-Osher-Fatemi  \cite{Rudin1992}, where they used the total variation as a regularizer, such variational models are being regularly used in imaging science. The key advantage of the fractional Laplacian regularizer from \cite{AntilBartels2017} is the fact that one arrives at the following \emph{linear} Euler-Lagrange equations:
\begin{equation}
	\label{eq:image_denoise}
	\flap u + \alpha (g-u) = 0.
\end{equation}
In the respective work, the authors use the spectral
fractional Laplacian which applies periodic boundary conditions. In contrast,
we use our method that uses Dirichlet-exterior value conditions. Here, we
subtract the mean $\bar g$ of $g$ from $g$ before the calculations and add it
back afterwards. The results can be seen in \cref{fig:g_denoised_u}. The goal
of this example is not to further illustrate the effectiveness of fractional
Laplacian as a regularizer, but to show that we can obtain comparable results
using the approach considered in this paper. This clearly follows from our
example.

\begin{figure}[h!]
	\begin{minipage}[c]{.49\textwidth}
		\begin{subfigure}{.49\textwidth}
				 \centering
			\includegraphics[width=\textwidth]{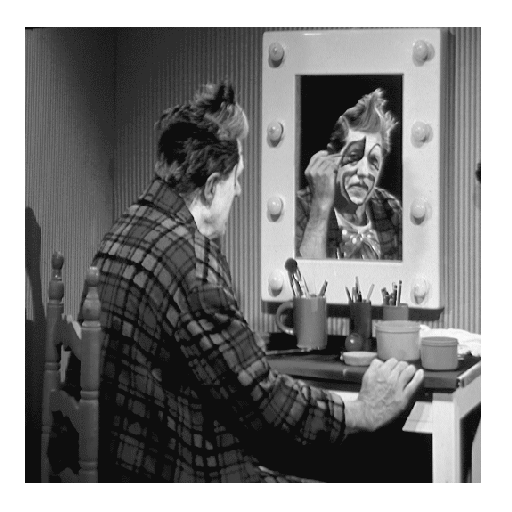}
			\caption{}
		\end{subfigure}
		\begin{subfigure}{.49\textwidth}
				 \centering
			\includegraphics[width=\textwidth]{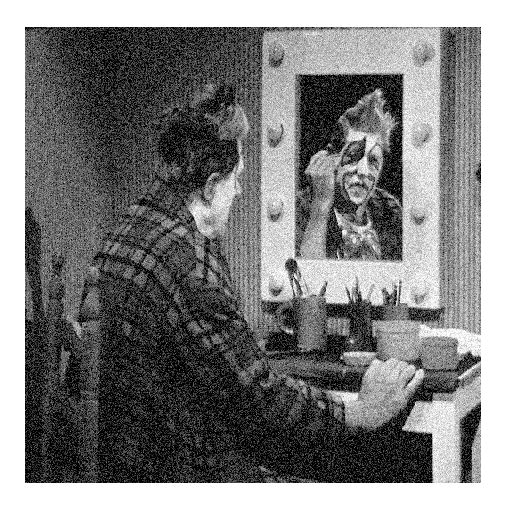}
			\caption{}
		\end{subfigure}
		\begin{subfigure}{.49\textwidth}
				 \centering
			\includegraphics[width=\textwidth]{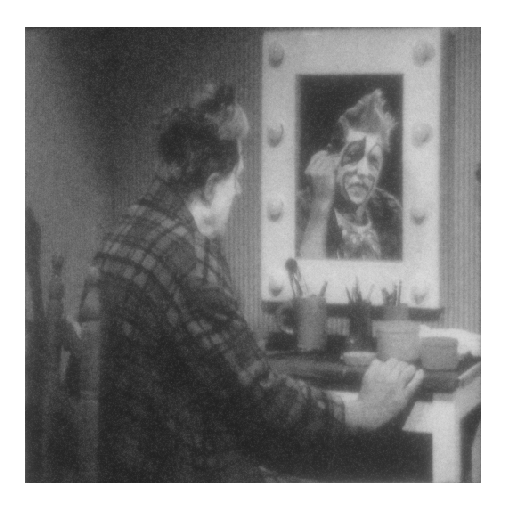}
			\caption{}
		\end{subfigure}
		\begin{subfigure}{.49\textwidth}
			\includegraphics[width=\textwidth]{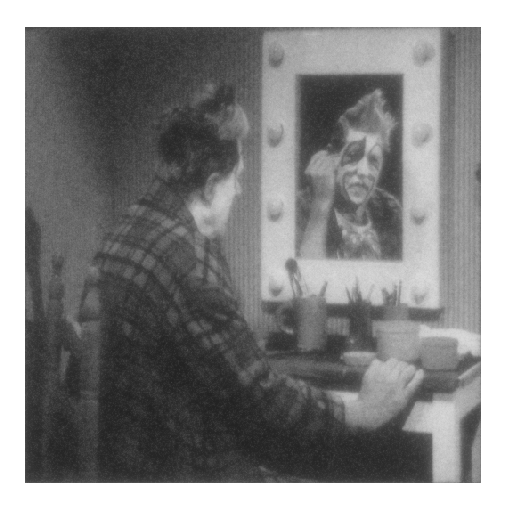}
			\caption{}
		\end{subfigure}
	\end{minipage}
	\begin{minipage}[c]{.49\textwidth}
		\begin{subfigure}{\textwidth}
			\includegraphics[width=\textwidth]{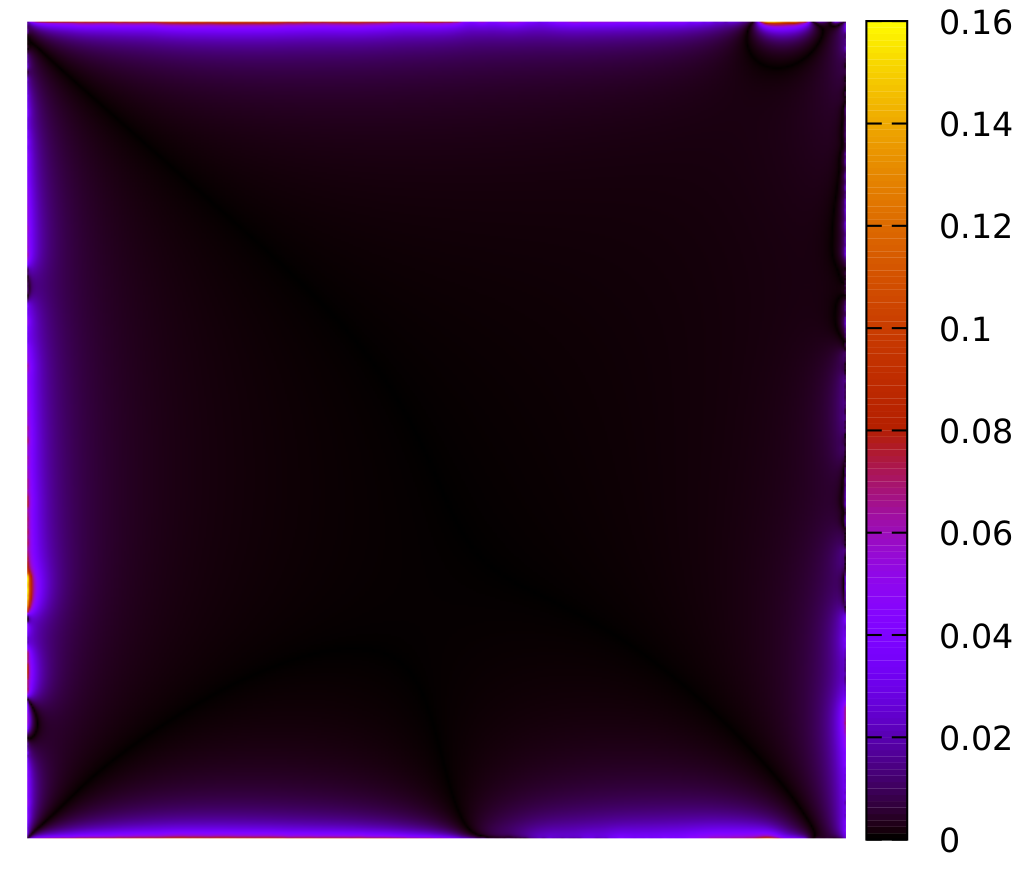}
			\caption{}
		\end{subfigure}
	\end{minipage}
	\caption{(a) Original image, (b) image corrupted with Gaussian noise, (c)
		denoised image using our method, and (d) the image denoised with the
		spectral method from \cite{AntilBartels2017}. We chose $s = 0.42$ and
		$\alpha = 10\cdot2\pi$. The difference of the image denoised with our
		method and denoised with the spectral method of \cite{AntilBartels2017}
		(see \cref{fig:g_denoised_u}) is shown in (e). The differences clearly
		concentrate at the boundary of the images as one would expect.
		}
	\label{fig:g_denoised_u}
\end{figure}

\section{Conclusion and Future Work}
The paper introduces a novel spectral method which allows efficient application of the fractional Laplacian in $\mathcal{O}(N\log(N))$ operations as well as a solution algorithm for fractional partial differential equations with Dirichlet exterior conditions. The proposed method works in both $2d$ and $3d$. We have further shown the effectiveness of the method in two applications: a fractional Allen-Cahn equation and an image denoising problem.  
The method works on arbitrary domains, for instance we have done computations on the ever popular L-shape domain. One potential limitation of our approach is that  
we can work only on uniform grids, a higher number of unknowns maybe required especially when the solution is expected to have
have singularities. Regardless, the application of our discrete operator still retains its $\mathcal{O}(N\log(N))$ complexity.

This work opens up new opportunities for problems where nonlocal operators such as fractional Laplacians appear, especially in $3d$. There are a number of open questions which are a matter of current investigation. In particular, 
\textit{(i)} How to extend the proposed method to other exterior conditions such as Neumann or Robin and how to handle nonzero exterior conditions; 
\textit{(ii)} a complete numerical analysis of the proposed method is currently missing; 
\textit{(iii)} we have applied the proposed method to both linear elliptic and nonlinear parabolic (Allen-Cahn) equations, it will also be interesting to carry out analysis in the nonlinear setting of Allen-Cahn; 
\textit{(iv)} it will be interesting to apply the proposed method to equilibrium problems such as variational inequalities and PDE constrained optimization problems.

\appendix
\section{Implementation details and computational complexity} \label{sub:algorithms}
In this section, we provide more details on the implementation of the algorithms.
The indices in the following computations will be chosen such that they fit the indices
of usual FFT implementations. All our implementations are written in \verb|C++| 
and rely on the FFTW library \cite{FFTW05}.
Computing the discrete solution of the Dirichlet-exterior value problem \cref{eq:bvp}
essentially consists of two steps,
\begin{itemize}
	\item[--] setup the discrete operator $\Phi^N$,
	\item[--] solve the (discrete) system $S_\Omega\Phi^N S_\Omega^T\u  = S_\Omega\f$.
\end{itemize}
Regarding the first step, we do not actually calculate $\Phi^N$, but its DFT $\hat\Phi^N$
following the procedure described in ~\cref{sub:setting_up_the_conv_kernel}. Further details are provided in ~\cref{alg:setup_phi_hat}.
\begin{algorithm}[h!]
	\caption{Calculation of the convolution kernel}\label{alg:setup_phi_hat}
	\begin{algorithmic}[1]
		\Function{calc\_PHI\_hat}{$N$, $s$, $d$}
			\State $\hat\Phi\gets $complex array of size $(2N)^d$, filled with $0$
			\For{$(x_i, \alpha_i)\in Q$}
				\State $c_1\gets $empty real array of size $(2N)^d$
				\State $c_1\gets $empty real array of size $()^d$
				\For {$j\in{\I}_{2N}$}
					\State $c_1[j]\gets |j-N\vec 1+x_i|^{2s}$
					\State $c_2[j]\gets Y_d(-\pi/N * (j-N\vec 1 + x_i))$
				\EndFor
				\State $C_1 \gets $ \Call{FFT}{$c_1$}\label{alg:setup_phi_hat:lfft1}
				\State $C_2 \gets $ \Call{FFT}{$c_2$}\label{alg:setup_phi_hat:lfft2}
				\State $C\gets C_1 * C_2$
				\State $c\gets$ \Call{IFFT}{$C$}\label{alg:setup_phi_hat:lfft3}
				\For {$k\in{\I}_{2N}$}
					\State $E_k\gets (2\pi)^{-d}* (\pi/N)^{d+2*s}*N^{2s}*\Call{exp}{\i\pi(k_1 + \cdots + k_d)}$
					\State $\hat\Phi[k]\gets\hat\Phi[k] + \alpha_i*E_k*c[k]$
				\EndFor 
			\EndFor
			\State\Return{$\hat\Phi$}
		\EndFunction
	\end{algorithmic}
\end{algorithm}
The arithmetic operations $*, +$ are meant component-wise if applied to arrays.
The computationally most demanding part in \cref{alg:setup_phi_hat} is
the 3-fold execution of the FFT-algorithm in lines
\ref{alg:setup_phi_hat:lfft1}, \ref{alg:setup_phi_hat:lfft2} and
\ref{alg:setup_phi_hat:lfft3} of \cref{alg:setup_phi_hat}, which are needed to compute the inner sum on the
right-hand side of \cref{eq:quadrature_rules_phi_hat}, i.e., the convolution.
These FFTs have to be executed for each of the quadrature points, i.e.,
$N_Q$-times and they are of size $(2N)^d$. However, this is done for all the
$k\in\I_{2N}^d$ at a total cost of $\mathcal O((2N)^d\log((2N)^d))$ which is
still small compared to a naive implementation at cost $\mathcal
O(((2N)^d)^2)$. 

For the second step, i.e., the solution of the fractional PDE, we use the conjugate gradient (CG) method. The method
solves a linear system
\[
	\mathbf A\u = \mathbf b,\quad \mathbf b\in\R^N, \mathbf A\in\R^{N\times N}
	\text{invertible symmetric and positive-definite}
\]
via successive applications of the matrix $\mathbf A$ instead of solving it directly.
Consequently, the algorithm is fast if the application of the
operator $\mathbf A$ can be computed efficiently. This is the case in our setting, since we need to compute the application (convolution) of $\Phi^N$ to a vector $\u$.
As mentioned above, we do so by using FFT based algorithms in order to
reduce the computational complexity. We have to introduce some padding in order to
apply the zero-padding convolution instead of circular convolution as the FFT based
algorithm would normally do. The details are provided in \cref{alg:apply_phi_hat}.
\begin{algorithm}[h!]
	\caption{Application of $\Phi^N$}\label{alg:apply_phi_hat}
	\begin{algorithmic}[1]
		\Function{apply\_PHI}{$\u$, $\hat\Phi^N$}
			\State $\bar\u\gets $complex array of size $(2N)^d$, filled with $0$
			\For {$k\in{\I}_{N}$}
				\State $\bar\u[k+N\vec 1] \gets\u[k]$
			\EndFor
			\State $\hat{\bar\u} \gets\Call{fft}{\bar u}$
			\For {$k\in{\I}_{2N}$}
				\State $\hat{\bar\u}[k] \gets \hat{\bar\u}[k] * \hat\Phi[k]$
			\EndFor
			\State $\bar\f\gets \Call{ifft}{\hat{\bar\u}}$
			\State $\f\gets $ empty real array of size $N^d$
			\For {$k\in{\I}_{N}$}
				\State $\f[k] \gets \bar\f[k]$
			\EndFor
			\State\Return{$\f$}
		\EndFunction
	\end{algorithmic}
\end{algorithm}

If the domain in the exterior value problem is the full cube $[0;1)^d$, we can simply
use the conjugate gradient method 
with the operator $\Phi^N$ using the efficient application described in
\cref{alg:apply_phi_hat}. If we want to restrict the exterior value
problem to a domain $\Omega\subsetneq[0;1)^d$, we use the strategy described in
\cref{sub:solving_the_system} and further summarized in \cref{alg:apply_phi_hat_omega}.
\begin{algorithm}[h!]
	\caption{Application of $\Phi$ with restriction to $\Omega\subset[0;1)^d$}\label{alg:apply_phi_hat_omega}
	\begin{algorithmic}[1]
		\Function{apply\_PHI\_Omega}{$\u$, $\hat\Phi^N$}
			\State $\mathbf S\gets $ empty real array of size $N^d$
			\For {$k\in{\I}_{N}$}
				\If{$k/N\in\Omega$}
					\State$\mathbf S[k]\gets 1$
				\Else
					\State$\mathbf S[k]\gets 0$
				\EndIf
			\EndFor
			\For {$k\in{\I}_{N}$}
				\State$\u[k] \gets \u[k]*\mathbf S[k]$
			\EndFor
			\State $\f\gets \Call{apply\_PHI}{\u, \hat\Phi^N}$
			\For {$k\in{\I}_{N}$}
				\State $\f[k] \gets \f[k]*\mathbf S[k]$
			\EndFor
			\State\Return{$\f$}
		\EndFunction
	\end{algorithmic}
\end{algorithm}
Both the operator $\Phi^N$ and the operator restricted to a smaller area can be applied to a vector $\u$ at cost $\mathcal{O}((2N)^d\log((2N)^d))$ which is substantially less then the cost $\mathcal{O}(((2N)^d)^2)$ of the naive
implementation.

\bibliographystyle{alpha}
\bibliography{bibliography}

\end{document}